\documentclass[hidelink,numbering]{myArt}

\title[Generic regularity of equilibrium measures]{Generic regularity of equilibrium measures\\ for the logarithmic potential with external fields}
\author{Giacomo Colombo and Alessio Figalli}\date{}

\address{ETH Z\"urich, Department of Mathematics, R\"amistrasse 101, 8092, Z\"urich, Switzerland}
 \email{giacomo.colombo@math.ethz.ch}  
 \email{alessio.figalli@math.ethz.ch}

\begin{document}

\begin{abstract}
A well-known conjecture in the theory of Coulomb gases and the mean-field limit of $\beta$-ensembles asserts that, generically, external potentials are ``off-critical'' (or equivalently, ``regular''). By exploiting the connection between minimizing measures and thin obstacle problems, we establish this conjecture.
\end{abstract}

\maketitle

\section{Introduction}
Given a potential $V:\R\to \R$, a central question arising both from the theory of Coulomb gases  and the mean-field limit of $\beta$-ensembles is to understand the behavior of probability measures $
\mu_V$ minimizing the energy
\begin{equation}\label{eq:energy}
    \mathcal{E}_V(\mu)\coloneqq2\int_\R V(x)\,d\mu(x)-\int_\R\int_\R\log|x-y|\,d\mu(x)\,d\mu(y)
\end{equation}
(see, for instance, \cite{saff_totik1997,serfaty2024lectures}).
The classical case $V(x)=x^2$, which comes from Gaussian Orthogonal Ensembles, was first studied in~\cite{wigner1958distribution}, and the minimizing probability is known as the ``semicircle law'', since its density $\frac{d\mu_V}{dx}=\frac2\pi\sqrt{(x+1)(1-x)}$ is, up to a constant, the graph of the (upper) unit circle.

\subsection*{\texorpdfstring{$\beta$}{}-models} Most results on $\beta$-models (see, for instance, \cite{anderson2010introduction,BG11,erdos2012local,bourgade2014edge,BFG15,FG16,BG_multicut24} and the references therein) rely on the assumption that the semicircle represents the general behavior of such minimizing measures. More precisely, given a potential $V,$ the following assumption (usually called ``off-criticality'' or ``regularity'' assumption) is made:

\smallskip
\noindent
{\bf (A)} {\it The  minimizing probability $\mu_V$ is supported over finitely many disjoint compact intervals and, inside each such interval $[a,b]$,
it has the form
\begin{equation}\label{eq:form-regular-potential}
    \frac{d\mu_V}{dx} = Q_V(x)\sqrt{(x-a)(b-x)},
\end{equation}
for some function $Q_V:\R\to \R$ satisfying $0<c \leq Q_V(x)\leq C$.}

\smallskip

Unfortunately, the property above is known to be false for arbitrary potentials. Still, it is conjectured to be ``generically'' true.
In the case of analytic potentials, this has indeed been shown in~\cite{kuijlaars2000generic}, but up to now, nothing was known in the non-analytic setting (except for some very special cases).

\smallskip

In this article, we give a positive answer to this conjecture for $C^{2,\alpha}$ potentials. Also, we show that higher regularity assumptions on $V$ yield higher regularity of $Q_V$.
Because it is well-known that the space of regular potentials is open, the challenge is to prove that they are dense.

Here and in the following, given $j \in \mathbb N$ and $\beta \in [0,1],$ we denote by $C_{\rm loc}^{j,\beta}(\R)$ the space of functions $g:\R\to \R$ that belong to $C^{j}(\R)$ and whose $j$-th derivative is locally $\beta$-H\"older continuous, that is
$$
\sup_{x\neq y \in [-R,R]}\frac{|D^jg(x)-D^jg(y)|}{|x-y|^\beta}<\infty\qquad \forall\,\,R>0.
$$
Our first main result is the following (see Theorem~\ref{teo:generic_continuousCase2} below for a more precise statement):

\begin{teo}\label{teo:generic_continuousCase}
    Given $\alpha \in (0,1)$, let $V\in C_{\rm loc}^{2,\alpha}(\R)$ satisfy $\lim_{|x|\to+\infty}\frac{V(x)}{\log|x|}=+\infty$. Given $\gamma\in \R$, consider the family of potentials $V_{s,\gamma}(x):=\frac{V(s^\gamma x)}{s}$, $s>0$. Then $V_{s,\gamma}$ is regular for a.e.  $s>0$.\\
    In particular, the set of potentials for which {\bf (A)} holds is open and dense in the class of $C_{\rm loc}^{2,\alpha}$ potentials that diverge at infinity faster than a logarithm.\\
    Furthermore, if {\bf (A)} holds and $V\in C_{\rm loc}^{k+1/2+\beta}(\R)$ for some $k\ge2$ and $\beta\in(0,1/2)\cup(1/2,1)$,\footnote{Here and in the sequel, for convenience of notation, we denote $$C_{\rm loc}^{k+1/2+\beta}(\R)=
\left\{
\begin{array}{ll}
C_{\rm loc}^{k,\beta+1/2}(\R)&\text{if }\beta \leq 1/2,\\
C_{\rm loc}^{k+1,\beta-1/2}(\R)&\text{if }\beta > 1/2.
\end{array}
\right.
$$}
then $Q_V\in C_{\rm loc}^{k-1,\beta}(\R)$.
\end{teo}

\begin{oss}
In the case of analytic potentials, this result with $\gamma=1$ recovers the one of~\cite{kuijlaars2000generic}.
\end{oss}

\subsection*{Discrete \texorpdfstring{$\beta$}{}-models}
A discrete version of the previous problem naturally arises, for instance, from asymptotics of distributions of discrete $\beta$-ensembles or in the study of orthogonal polynomials of a discrete variable, see~\cite{borodin2017gaussian,dragnev_saff1997}.

In this case, given $N\in\N$ and $\theta>0$, the energy~\eqref{eq:energy} is minimized among probability measures with density bounded above by $\theta$ and supported inside a compact set $K:=\bigcup_{h=1}^N[a_h,b_h]$, where the previous intervals are disjoint. We always assume $\theta|K|>1$.
In addition, as considered in \cite{borodin2017gaussian}, one prescribes the mass of the admissible measures $\mu$ inside each interval:
\begin{equation}
    \label{eq:mass h}
    \mu([a_h,b_h])=\hat n_h \geq 0,\qquad \sum_{h=1}^N \hat n_h=1,\qquad 0<\hat n_h<\theta|b_h-a_h|.
\end{equation}
Furthermore, the potential $V$ is differentiable inside $\bigcup_{h=1}^N[a_h,b_h]$ and is assumed to satisfy
\begin{equation}\label{eq:derivativeBound-discrete}
    |V'(x)|\le C\left(1+\sum_{h=1}^N\big|\log |x-a_h|\big|+\big|\log|x-b_h|\big|\right)
\end{equation}
for some constant $C>0$.

In analogy to the continuous case, an important assumption on the minimizing measure is the following (see, for instance, \cite{borodin2017gaussian}):

\smallskip
\noindent

{\bf (B)} {\it The set $\{0<\frac{d\mu_{V,\theta}}{dx}<\theta\}$ is a finite union of intervals compactly contained inside $\bigcup_{h=1}^N(a_h,b_h)$. In addition, the  density $\psi=\frac{d\mu_{V,\theta}}{dx}$ vanishes like a square root near $\partial\{\psi>0\}$ and converges to $\theta$ like a square root near $\partial\{\psi<\theta\}$. More precisely, given $p_-\in\partial\{\psi >0\}$ (resp. $p_+\in\partial\{\psi <\theta\}$) there exists a uniformly positive, bounded function $Q_-$ (resp. $Q_+$), defined in a neighborhood of $p_-$ (resp. $p_+$), such that
\begin{equation}\label{eq:form-regular-potential-discrete}\begin{split}
    &\frac{d\mu_{V,\theta}}{dx}(x) = Q_-\sqrt{|x-p_-|}\text{ for $x\in \left\{0<{\textstyle \frac{d\mu_{V,\theta}}{dx}}<\theta\right\}$, $|x-p_-|\ll 1$,}\\
&\text{(resp. }   \frac{d\mu_{V,\theta}}{dx}=\theta- Q_+\sqrt{|x-p_+|}\text{ for $x\in \left\{0<{\textstyle \frac{d\mu_{V,\theta}}{dx}}<\theta\right\}$, $|x-p_+|\ll 1$)}.
\end{split}\end{equation}
}

\smallskip

Given $N$ disjoint intervals $[a_h,b_h]$, $j\in\N$, and $\beta\in[0,1]$, we denote by $C^{j,\beta}_\loc\left(\bigcup_h (a_h,b_h)\right)$ the space of functions $g\colon\bigcup_h (a_h,b_h)\to\R$ that belong to $C^j\left(\bigcup_h (a_h,b_h)\right)$ and whose $j$-th derivative is locally $\beta$-H\"older continuous, that is
$$
\sup_{x\neq y \in \bigcup_h[-a_h+r,b_h-r]}\frac{|D^jg(x)-D^jg(y)|}{|x-y|^\beta}<\infty\qquad \forall\,\,r>0\text{ sufficiently small}.
$$

Here is our main result (see Theorem~\ref{teo:generic_discreteCase2} below for a more precise statement, in particular concerning the topology considered for our openness/denseness statement):

\begin{teo}\label{teo:generic_discreteCase}
Given $\alpha \in (0,1)$, let $V\in C_{\rm loc}^{2,\alpha}\left(\bigcup_h (a_h,b_h)\right)$ satisfy~\eqref{eq:derivativeBound-discrete}.
Consider the family of potentials $V_s(x):=\frac{V(s x)}{s}$, $0<s<\theta|K|$, and let $\mu_s$ be the probability measure minimizing $\cE_{V_s}$ with density bounded by $\theta$ (without additional constraints).
Then $\mu_s$ is regular for a.e.  $s\in(0,\theta|K|)$.
In particular, the set of potentials for which {\bf (B)} holds is open and dense in the class of $C_{\rm loc}^{2,\alpha}$ potentials satisfying~\eqref{eq:derivativeBound-discrete}.\\
In addition, for $s\in(0,\theta|K|)$, define the $(N-1)$-dimensional convex sets
$$
\mathcal N_s:=\bigg\{\mathbf n=(\hat n_1,\ldots,\hat n_N) \in \R^N\,:\,\sum_{h=1}^N \hat n_h=1,\qquad 0<\hat n_h<\theta\frac{|b_h-a_h|}{s}\bigg\}.
$$
Then, for a.e. $s\in(0,\theta|K|)$ and $\mathcal H^{N-1}$-a.e. $\mathbf n \in \mathcal N_s$, the probability measure $\mu_{s,\mathbf n}$ minimizing $\cE_{V_s}$ with constraint $\mathbf n$ and with density bounded by $\theta$ is regular.\\
Furthermore, if {\bf (B)} holds and $V\in C^{k+1/2+\beta}_{\rm loc}\left(\bigcup_h(a_h,b_h)\right)$ for some $k\ge2$ and $\beta\in(0,1/2)\cup(1/2,1)$,\footnote{As before, for convenience of notation, given $U\subset \R$ open we denote $$C_{\rm loc}^{k+1/2+\beta}(U)=
\left\{
\begin{array}{ll}
C_{\rm loc}^{k,\beta+1/2}(U)&\text{if }\beta \leq 1/2,\\
C_{\rm loc}^{k+1,\beta-1/2}(U)&\text{if }\beta > 1/2.
\end{array}
\right.
$$
}
then the functions $Q_\pm$ are of class $C_{\rm loc}^{k-1,\beta}$.
\end{teo}
\begin{oss}
    Note that the classes of potentials appearing in \cite[Eq. (120,122)]{borodin2017gaussian} satisfy \eqref{eq:derivativeBound-discrete} and are closed (up to global constants that can be neglected due to the mass constraint) with respect to the 1-homogeneous rescaling. Thus, our result implies:
    \begin{itemize}
        \item in the case of no additional mass constraints (as in \cite[Eq. (120)]{borodin2017gaussian}), for almost every choice of the parameters, the potentials are regular;
    \item in the case of additional mass constraints (as in \cite[Eq. (122)]{borodin2017gaussian}, for almost every choice of the mass constraints and of the parameters, the corresponding minimizing measures are regular.
    \end{itemize}
\end{oss}

\subsection*{Riesz potentials} One can also consider more general energies of the form
\[
    \cE(\mu)=\iint_{\R^d\times\R^d}\mathfrak g(x-y)\,d\mu(x)\,d\mu(y) + 2\int_{\R^d}V(x)\,d\mu(x),
\]
where $d \geq 1$ and
\begin{equation}\label{eq:g}
\mathfrak g(x) = \begin{cases}
    \frac1{\sigma}|x|^{-\sigma} &\text{if } \sigma\neq0,\\
    -\log|x| &\text{if }\sigma=0,
\end{cases}\qquad\text{with}\quad \sigma \in (d-2,d).
\end{equation}
These energies arise from the study of Riesz gases (see for instance~\cite{serfaty2024lectures} and references therein).
In this context, the analogue of assumption \textbf{(A)} is the following:

\smallskip
\noindent
{\bf ($\mathbf{A_\sigma}$)} {\it The  minimizing probability $\mu_V$ is supported over finitely many disjoint compact sets $\{K_j\}_{1\leq j \leq M}\subset \R^d$, with $\partial K_j$ a $(d-1)$-dimensional manifold of class $C^{1}$. Also, inside $K_j$,
it has the form
\begin{equation}\label{eq:form-regular-potential sigma}
    \frac{d\mu_V}{dx} = Q_V(x)\,{\rm dist}(x,\partial K_j)^{1-\frac{d-\sigma}2},
\end{equation}
for some function $Q_V:\R^d\to \R$ satisfying $0<c \leq Q_V(x)\leq C$.}

\smallskip

As discussed in Appendix~\ref{app:riesz case},
our strategy also applies in this case and proves the generic validity of {\bf ($\mathbf{A_\sigma}$)} in low dimensions (see
Theorem~\ref{teo:generic_riesz}).

\subsection{Comments on the proofs and structure of the paper}
To prove Theorem~\ref{teo:generic_continuousCase},
we adopt Serfaty's strategy 
to associate to $\mu_V$ a solution to the thin obstacle problem (see \cite{Serfaty2015,LS17,serfaty2024lectures,peilen_serfaty2024local} and Section~\ref{sect:min to thin} below).
In this way, one may hope to apply the regularity theory available for generic solutions of the thin obstacle problem \cite{FERNANDEZREAL2023109323}. Note that the latter theory applies only to a strictly monotone family of solutions, so one would need to find perturbations of $V$ that induce monotone perturbations of the solutions. 
While this seems difficult to achieve, we follow an idea from~\cite{kuijlaars2000generic} where, instead of changing the potential, we vary the mass constraint. 
This variation yields a monotone family of solutions (see Proposition~\ref{prop:monotonicity_continuousCase}), allowing the application of the generic regularity theory for the thin obstacle problem. Finally, by using scaling properties of minimizers of $\mathcal{E}_V$, we obtain the desired result (see Remark~\ref{rmk:rescaling_continuous}).

\smallskip

While in the case of $\beta$-models we can rely on the available theory for the thin obstacle problem, the discrete case (Theorem~\ref{teo:generic_discreteCase}) is considerably more challenging. In this setting, too, the regularity of the minimizing measure can be related to the regularity of a solution of a PDE. However, the additional constraint on the upper bound of the measure’s density forces the solution to satisfy a two-phase thin obstacle problem (see Theorem~\ref{teo:teo:limitmeasure_discrete}), for which no generic regularity theory is currently available.
A key result established in Appendix~\ref{sec:phase separation} addresses this difficulty: in the two-phase thin obstacle problem, phases with opposite signs cannot touch. This result is not only central to our analysis, but also provides a fundamental tool for studying the free boundary in cohesive zone models for fracture mechanics (cf.~\cite{Caffarelli2020}), and it extends a related result from \cite{Allen15}.
Another challenge arises from the non-smoothness of $V$ at the boundary of its support, which may create singularities in our solutions. These issues are resolved in Appendices~\ref{sec:phase separation} and \ref{sec:double top}. In particular, the results proved in Appendix~\ref{sec:phase separation} fix a gap in \cite{Caffarelli2020} and are of independent interest.

\smallskip

Since this paper uses and develops nontrivial PDE tools but has applications beyond the PDE community, we structure the paper as follows: In Section~\ref{sec:min obst}, following Serfaty's argument, we show the connection between minimizing measures and obstacle problems. Then, in Sections~\ref{sec:proof thm1} and~\ref{sec:proof thm2} we prove our main results, taking for granted the generic regularity properties of solutions to obstacle problems.
Finally, all PDE materials (both the known theory for the classical thin obstacle problem and the new theory needed for the case of discrete models) are postponed to Appendices~\ref{sect:TOP}, \ref{sec:phase separation}, and~\ref{sec:double top}.
In Appendix~\ref{app:riesz case}, we briefly discuss the case of Riesz potentials.

\smallskip\noindent\textit{Acknowledgments:} We thank Alice Guionnet for suggesting this problem to us and for pointing out the connection, due to Sylvia Serfaty, between minimizing measures and obstacle problems. AF is grateful to the Marvin V. and Beverly J. Mielke Fund for supporting his stay at IAS Princeton, where part of this work has been done. 

\smallskip\noindent\textit{Data Availability:} There is no data associated to this work

\smallskip\noindent\textit{Conflict of Interest:} The authors have no Conflict of interest to disclose

\section{From minimizing measures to obstacle problems}\label{sec:min obst}
The goal of this section is to discuss some preliminary facts about measures minimizing the energy \eqref{eq:energy} and show the connection to the thin obstacle problem.
This will allow us to reformulate our main theorems in a slightly different form, see Theorems~\ref{teo:generic_continuousCase2}
and~\ref{teo:generic_discreteCase2} below.
The proofs of these two theorems will then be given in the next two sections.

We begin this section by discussing some preliminary facts about measures minimizing the energy \eqref{eq:energy}.

\subsection{From minimizing measures to the thin obstacle problem}\label{sect:min to thin}
Here and in the sequel, we denote by $\cM_1(\R)$ the space of nonnegative probability measures on $\R$.
We begin with the non-discrete case.
\begin{teo}[{\cite[Lemma 2.6.2]{anderson2010introduction}}]\label{teo:limitmeasure}
    Let $V:\R\to \R$ be a continuous potential satisfying
    \[\lim_{|x|\to\infty}\frac{V(x)}{\log|x|}=+\infty.\]
    Then there is a unique probability measure $\mu_V$ minimizing
    \[\inf\{\mathcal{E}_V(\nu)\,:\,\nu\in\cM_1(\R)\}.\]
    In addition, a probability $\mu$ is minimizing if and only if there is a constant $C\in \R$ such that
    \begin{equation}\label{eq:ele}-\int_\R\log|x-t|\,d\mu(t)\ge C-V(x)\quad\text{ for every }x\in\R,\text{ with equality }\mu\text{-a.e}.\end{equation}
\end{teo}

To connect minimizing measures to the thin obstacle problem, given the measure $\mu_V$ minimizing $\mathcal E_V$ in $\cM_1(\R)$, consider the function
    \begin{equation}\label{eq:potentials}u_V(x)=-\big(\log|\cdot|*\mu_V\big)(x)- C_V,\quad x\in\R^2.\end{equation}
Then, as discussed for instance in \cite[Chapter 2.4]{serfaty2024lectures}, \eqref{eq:ele} is equivalent to asking that $u_V$ solves the thin obstacle problem with obstacle $-V$ (cf. Appendix~\ref{sect:TOP}):
\begin{equation}\label{eq:eleTop}\begin{cases}
    -\Delta u_V\ge0 &\text{in }\R^2,\\
    \Delta u_V = 0  &\text{in }\R^2\setminus(\{u_V = -V\}\cap\{x_2=0\}),\\
    u_V\ge -V       &\text{in }\{x_2=0\},\\
    \frac{u_V(x)}{\log|x|}\to -1&\text{as }|x|\to+\infty.
\end{cases}\end{equation}
This motivates the following definition (here and in the sequel, $\cH^k$ denotes the $k$-dimensional Hausdorff measure).

\begin{defn}[Regular potential]\label{defn:regular-potential-continuous}
We say that a potential $V$ is \emph{regular} if:
\begin{itemize}
    \item[\textit{i)}] $\{u_V = -V\}=\supp\mu_V$;
    \item[\textit{ii)}] $\supp\mu_V$ is a finite union of disjoint compact intervals;
    \item[\textit{iii)}] $\mu_V\ll\cH^1$ and its density is of the form~\eqref{eq:form-regular-potential} in each interval $[a,b]$ of $\supp\mu_V$,
\end{itemize}
    where $u_V$ is defined by \eqref{eq:potentials}.
\end{defn}

We define next the topology under consideration when discussing an open and dense set.
Consider the set
\[
X\coloneqq\Big\{V\in C^{2,\alpha}_{\rm loc}(\R)\,:\,\lim_{|x|\to+\infty}\frac{V(x)}{\log|x|}=+\infty\Big\}
\]
endowed with the distance\footnote{In \cite{kuijlaars2000generic}, the authors consider an analogous distance but in the space of $C^{3}$ potentials.}
\[\begin{split}
    \rho(V,W):=\sum_k 2^{-k}\frac{\|V-W\|_{C^{2,\alpha}(-k,k)}}{1+\|V-W\|_{C^{2,\alpha}(-k,k)}} +\sum_k 2^{-k}\frac{|G_k(V)-G_k(W)|}{1+|G_k(V)-G_k(W)|},
\end{split}\]
where
\[
    G_k\colon X\to\R,\qquad G_k(V)\coloneqq \inf_{|x|\ge k}\frac{V(x)}{\log|x|}.
\]
Notice that $\rho(V_k,V)\to 0$ if and only if $V_k\to V$ in $C^{2,\alpha}_\loc(\R)$ and $\lim_{|x|\to+\infty} \frac{V_k(x)}{\log|x|}=+\infty$ uniformly in $k$.

We can now state our refined version of Theorem~\ref{teo:generic_continuousCase}.

\begin{teo}\label{teo:generic_continuousCase2}    
    Given $\alpha \in (0,1)$, let $V\in C_{\rm loc}^{2,\alpha}(\R)$ satisfy $\lim_{|x|\to+\infty}\frac{V(x)}{\log|x|}=+\infty$. Given $\gamma\in \R$, consider the family of potentials $V_{s,\gamma}(x):=\frac{V(s^\gamma x)}{s}$, $s>0$.
    Then $V_{s,\gamma}$ is regular for a.e. $s>0$ in the sense of Definition~\ref{defn:regular-potential-continuous}.
    In particular, the set of regular potentials is an open and dense subset of $(X,\rho)$.\\
    Furthermore, if $V:\R\to \R$ is regular and belongs to $C_{\rm loc}^{k+1/2+\beta}(\R)$ for some $k\ge2$ and $\beta\in(0,1/2)\cup(1/2,1)$, then the function $Q_V$ in~\eqref{eq:form-regular-potential} is of class $C_{\rm loc}^{k-1,\beta}(\R)$.
\end{teo}

\subsection{From minimizing measures to the two-phase thin obstacle problem: the discrete case}
Given $K=\bigcup_{h=1}^N[a_h,b_h]$ a finite union of disjoint intervals, we denote
\[
    \cM_{1,\theta}(K)\coloneqq\left\{\nu=\eta\cH^1\in\cM_{1}(\R)\,:\,0\le \eta\le\theta,\quad \supp\nu\subset K\right\}.
\]
We assume that $\theta|K|>1$, so that $\cM_{1,\theta}(K)$ is nontrivial.
In addition, given nonnegative numbers $\hat n_h$ satisfying \eqref{eq:mass h}, we set $\mathbf n=(\hat n_1,\ldots,\hat n_N) \in \R^N$ and 
\[
    \cM_{\mathbf n,\theta}(K)\coloneqq\left\{\nu = \eta\cH^1 \in \cM_1(\R)\,:\, 0\le\eta\le\theta,\quad\supp\nu\subseteq K,\quad \nu([a_h,b_h])=\hat n_h\right\}.
\]
\begin{teo}\label{teo:teo:limitmeasure_discrete}
    For every $\theta>0$ satisfying $\theta|K|>1$ there exists a unique probability minimizing
    \[
        \inf\{\mathcal{E}_V(\nu)\,:\,\nu\in\cM_{1,\theta}(K)\}.
    \]
   Moreover, a probability $\mu$ is minimizing if and only if there is a constant $C$ such that the function
\begin{equation}\label{eq:potentials2}
    u(x)=-\big(\log|\cdot|*\mu\big)(x)- C,\quad x\in\R^2,
\end{equation}
    solves
\begin{equation}\label{eq:ele_discrete}
    \begin{cases}
        \Delta u = 0      &\text{in }\R^2\setminus K,\\
        -\pi\theta\le\partial_2u\le0  &\text{in }K,\\
        \partial_2u=0     &\text{in }K\cap\{u>-V\},\\
        -\partial_2u=\pi\theta   &\text{in }K\cap\{u<-V\},\\
        \frac{u(x)}{\log|x|}\to -1   &\text{as }|x|\to+\infty,
    \end{cases}
\end{equation}
where $u$ is even in $x_2$ and the value $\partial_2u$ at $\{x_2=0\}$ is intended as the limit from the right, namely $$\partial_2u(x_1,0)=\lim_{t\to0^+}\frac{u(x_1,t)-u(x_1,0)}{t}.$$
Similarly, given $\mathbf n=(\hat n_1,\ldots,\hat n_N)$ satisfying \eqref{eq:mass h}, there exists a unique probability minimizing
\[
    \inf\{\mathcal{E}_V(\nu)\,:\,\nu\in\cM_{\mathbf n,\theta}(K)\}.
\]
In addition, a measure $\mu\in\cM_{\mathbf n,\theta}(K)$ is minimizing if and only if there are constants $\{C_h\}_{h=1}^N$ such that, defining $u$ as in \eqref{eq:potentials2} with $C=0$, it solves 
    \begin{equation}\label{eq:ele_discrete mass}
    \begin{cases}
        \Delta u = 0      &\text{in }\R^2\setminus K,\\
        -\pi\theta\le\partial_2u\le0  &\text{in }K,\\
        \partial_2u=0     &\text{in }\{x_2=0,a_h\le x_1\le b_h\}\cap\{u>-V-C_h\},\\
        -\partial_2u=\pi\theta   &\text{in }\{x_2=0,a_h\le x_1\le b_h\}\cap\{u<-V-C_h\},\\
        \frac{u(x)}{\log|x|}\to -1   &\text{as }|x|\to+\infty.
    \end{cases}
\end{equation}
\end{teo}
\begin{proof}
    The existence and uniqueness of a minimizing probability $\mu$ is shown in~\cite[Lemma 5.1]{borodin2017gaussian}.
    The proof of the necessary and sufficient conditions \eqref{eq:ele_discrete} and \eqref{eq:ele_discrete mass} follows as in~\cite[Lemma 5.5]{borodin2017gaussian}.
\end{proof}

In analogy to the previous subsection, we give the following definition.

\begin{defn}[Regular potentials, discrete case]\label{defn:regular-potential-discrete}
    Given $\theta>0$ satisfying $\theta|K|>1$, we say that a potential $V$ is \emph{regular} if, denoting by $\psi$ the density of the minimizing measure, it holds:
    \begin{itemize}
        \item[\textit{i)}] $\{u = -V\}=\{0<\psi<\theta\}$;
        \item[\textit{ii)}] $\{0<\psi<\theta\}$ is a finite union of open intervals and is compactly contained in the interior of $\bigcup_{h=1}^N(a_h,b_h)$;
        \item[\textit{iii)}] $\psi$ (resp., $\theta-\psi$) is of the form~\eqref{eq:form-regular-potential-discrete} at each point $p_-\in\partial\{\psi>0\}$ (resp. $p_+\in\partial\{\psi<\theta\}$),
    \end{itemize}
    where $u$ is defined by \eqref{eq:potentials2}.
\end{defn}
The definition of regular potential in the case with additional mass constraints is completely analogous.

\smallskip

To state our refined version of Theorem~\ref{teo:generic_discreteCase}, we introduce a topology on the set
\[
    \widetilde X \coloneqq\left\{(V,K)\mid K\subset \R\text{ is of the form }\cup_{k=1}^N[a_h,b_h]\text{ and }V\in C^{2,\alpha}_\loc(K)\text{ satisfies~\eqref{eq:derivativeBound-discrete}}\right\}
\]
induced by the distance
\begin{multline*}\widetilde\rho((V,K),(W,C))\coloneqq \mathrm{dist}_{\mathrm{H}}\big(\graph(V_{|_K}),\graph(W_{|_C})\big)\\
    +\sum_{n\in\N}2^{-n}\sum_{h=1}^N \frac{\|V-W\|_{C^{2,\alpha}((K\cap C)_{1/n})}}{1+\|V-W\|_{C^{2,\alpha}((K\cap C)_{1/n})}},
\end{multline*}
where $\mathrm{dist}_{\mathrm{H}}$ denotes the Hausdorff distance between two sets, and given $E\subset\R$ and $\rho>0$, we denote $E_\rho=\{x\in E\,:\,\dist(x,E^c)>\rho\}$.

Note that $\widetilde\rho((V_k,K_k),(V,K))\to 0$ if and only if $V_k\to V$ in $C^{2,\alpha}_\loc(K)$ and the graphs of $V_k$ on $K_k$ converge in the Hausdorff distance to the graph of $V$ on $K$.

\begin{teo}\label{teo:generic_discreteCase2}
Given $\alpha \in (0,1)$, let $V\in C_{\rm loc}^{2,\alpha}\left(\bigcup_h (a_h,b_h)\right)$ satisfy~\eqref{eq:derivativeBound-discrete}.
Consider the family of potentials $V_s(x):=\frac{V(s x)}{s}$, $s\in(0,\theta|K|)$.
Then $V_s$ is regular in the sense of Definition~\ref{defn:regular-potential-discrete} for a.e.  $s\in(0,\theta|K|)$.
In particular, the set of regular potentials is an open and dense subset of $(\widetilde X,\widetilde \rho)$.\\
In addition, for $s\in(0,\theta|K|)$, define the $(N-1)$-dimensional convex sets
$$
\mathcal N_s:=\bigg\{\mathbf n=(\hat n_1,\ldots,\hat n_N) \in \R^N\,:\,\sum_{h=1}^N \hat n_h=1,\qquad 0<\hat n_h<\theta\frac{|b_h-a_h|}{s}\bigg\}.
$$
Then, for a.e. $s\in(0,\theta|K|)$ and $\mathcal H^{N-1}$-a.e. $\mathbf n \in \mathcal N_s$, the probability measure $\mu_{s,\mathbf n}$ minimizing $\cE_{V_s}$ with constraint $\mathbf n$ is regular in the sense of Definition~\ref{defn:regular-potential-discrete}.\\
Furthermore, if $V:\bigcup_h[a_h,b_h]\to \R$ is regular and belongs to $C_{\rm loc}^{k+1/2+\beta}(\bigcup_h(a_h,b_h))$ for some $k\ge2$ and $\beta\in(0,1/2)\cup(1/2,1)$,
then the functions $Q_\pm$ are of class $C^{k-1,\beta}_{\rm loc}$.
\end{teo}

\section{Proof of Theorem~\ref{teo:generic_continuousCase2}}\label{sec:proof thm1}

As mentioned in the introduction, the main difficulty is to prove the density of regular potentials.
We will achieve this by showing that, given $\gamma\in\R$ and a potential $V\in C^{2,\alpha}_{\rm loc
}(\R)$ satisfying $\lim_{|x|\to+\infty}\frac{V(x)}{\log|x|}=+\infty$ and $\gamma\in\R$, the potential $V_{s,\gamma}\coloneqq \frac{V(s^\gamma\cdot)}{s}$ is regular for a.e.  $s>0$.

Here and in the sequel, we denote by $\cM_s(\R)$ the space of nonnegative measures on $\R$ with total mass $s>0$. The following observation, already used in  the proof of generic regularity for analytic potentials in~\cite{kuijlaars2000generic} for $\gamma=0$, relates scaling the potential $V_{s,\gamma}$ to rescaling the mass.
\begin{oss}\label{rmk:rescaling_continuous}
Given $s>0$, denote $\rho_s(x):=sx$.
Then, for every $\gamma\in\R$ and $s>0$, a measure $\mu$ minimizes $\{\mathcal{E}_V(\nu)\,:\,\nu\in\cM_s(\R)\}$ if and only if $\frac1s (\rho_{s^{-\gamma}})_\#\mu$ minimizes $\{\mathcal{E}_{V_{s,\gamma}}(\nu)\,:\,\nu\in\cM_1(\R)\}$\footnote{Given measurable spaces $X,Y$, a measure $\mu$ on $X$, and a measurable function $f\colon X\to Y$, we denote by $f_\#\mu$ the push-forward of $\mu$ through $f$, namely the measure on $Y$ defined by $f_\#\mu(A)=\mu(f^{-1}(A))$ for every measurable set $A\subset Y$.}.
\end{oss}

Given a potential $V$ and $\gamma\in\R$, for every $s>0$
let $\mu_s$ be the measure minimizing $\{\mathcal{E}_V(\nu)\,:\,\nu(\R)=s\}$,\footnote{Note that
Theorem~\ref{teo:limitmeasure} applies, with the same proof, to any value of $s>0$.} and consider the function \begin{equation}\label{eq:potentials s}u_s(x)=-\log|\cdot|*\mu_s(x) + \gamma s\log s-sC_{V_{s,\gamma}},\quad x\in\R^2,\end{equation}
    where $C_{V_{s,\gamma}}$ is the constant associated to the probability measure minimizing $\cE_{V_{s,\gamma}}$ in $\cM_1(\R)$.
Then, applying Theorem~\ref{teo:limitmeasure} with $V_{s,\gamma}$ and Remark~\ref{rmk:rescaling_continuous}, it follows that $u_s$ solves the following thin obstacle problem with obstacle $-V$:
\begin{equation}\label{eq:eleTop s}\begin{cases}
    -\Delta u_s\ge0 &\text{in }\R^2,\\
    \Delta u_s = 0  &\text{in }\R^2\setminus\{u_V = -V\}\cap\{x_2=0\},\\
    u_s\ge -V       &\text{in }\{x_2=0\},\\
    \frac{u_s(x)}{\log|x|}\to -s&\text{as }|x|\to+\infty.
\end{cases}\end{equation}

The following result is crucial for us.

\begin{prop}\label{prop:monotonicity_continuousCase}
    The functions $u_s$ are decreasing in $s$, namely $u_s\ge u_{s'}$ for any $s<s'$.
    More precisely, they are strictly decreasing in the following sense: for every $M>0$ there is $R_0$ such that for any $R\ge R_0$ there is $a=a(R)>0$ such that
    \begin{equation}\label{eq:strictMonotonicity-continuousCase}
    u_{s-\delta}-u_s > a\delta\quad\text{in }\partial B_R\cap\{|x_2|>R/2\}\end{equation}
    for every $0< s-\delta\le s\le M$ and $R\ge R_0$.
\end{prop}

    A proof of the monotonicity is essentially contained in~\cite[Theorem 2]{Buyarov1999}, using potential theory. Here we give an alternative proof using the comparison principle for~\eqref{eq:eleTop}, based on the following remark.
\begin{oss}\label{rmk:comparisonPrinciple-top}
    We state the comparison principle for~\eqref{eq:eleTop} allowing for different potentials, since we will need it later.
    Let $u,v$ be two solutions of~\eqref{eq:eleTop} in $B_1$ with potentials $V_u,V_v$. Assume that $V_u\le V_v$ in $\supp \Delta u$.
    Then $(u-v)_+$ is subharmonic in $B_1$.
    Indeed:\\
    - $\Delta (u-v)=-\Delta v\ge0$ in $\R^2\setminus\supp(\Delta u)$;
\\
- $u\le v$ in $\supp (\Delta u)$ (since, for $x\in\supp(\Delta u)$, we have $u(x)=V_u(x)\le V_v(x)\le v(x)$).\\
    Thus $\Delta (u-v)_+\ge0$ in $B_1$.
\end{oss}

For convenience of notation, we write $u \sim s\log|\cdot|$ to denote that $\lim_{|x|\to \infty} \frac{u(x)}{s\log|x|}\to 1$.

\begin{proof}[Proof of Proposition~\ref{prop:monotonicity_continuousCase}]
    Given $s<s'$, we note that
    $$u_s(x)\sim -s\log|x| \gg  -s'\log|x|\sim u_{s'}(x)\quad\text{ as $x\to+\infty$.}
    $$
    Hence, for any large enough ball $B_R$, we have $u_{s'}\le u_{s}$ in $\partial B_R$. Thus $(u_{s'}-u_{s})_+$ vanishes on $\partial B_R$, is nonnegative, and (by Remark~\ref{rmk:comparisonPrinciple-top}) it is subharmonic in $B_R$. By the maximum principle, this implies $u_{s'}\le u_{s}$ in $B_R$.
    Since $R$ can be taken arbitrarily large, we conclude that $u_s\ge u_{s'}$ in $\R^2$.

    Let us now show~\eqref{eq:strictMonotonicity-continuousCase}.
    Since $\supp \Delta u_s \subset\{u_s=-V\}$ and the functions are monotone, for every $M$ there is $R_0$ such that $\{u_s=-V\}\subset B_{R_0/2}$ for every $s\le M$.
    Thus we have
\[
    \Delta (u_{s-\delta}-u_s)=0     \quad\text{in }\R^2\setminus \{x_2=0,|x_1|\le R_0/2\},
\]
$u_{s-\delta}-u_s\ge0$ (by the previous step), and $u_{s-\delta}-u_s\sim \delta\log|\cdot|$ as $|x|\sim+\infty$.
Hence, if $\eta$ is a solution of the problem
\[\begin{cases}
    \Delta \eta =0  &\text{in }\R^2\setminus[-R_0/2,R_0/2]\times\{0\},\\
    \eta=0          &\text{in }[-R_0/2,R_0/2]\times\{0\},\\
    \eta\sim \log|\cdot|    &\text{as }|x|\to+\infty,
\end{cases}\]
the maximum principle implies $u_{s-\delta}-u_s\ge \delta\eta$ in $\R^2$. To conclude, we can take $a(R):=\inf_{\partial B_R}\eta$, which is positive by the strong maximum principle.
\end{proof}
We can now prove our main result.

\begin{proof}[Proof of Theorem~\ref{teo:generic_continuousCase2}]
We split the proof into 3 steps.

\noindent
    \textbf{Step 1: Regular potentials are dense.}
    Given $V\in C^{2,\alpha}_\loc(\R)$ satisfying $\lim_{|x|\to+\infty} \frac{V(x)}{\log|x|}=+\infty$ and $\gamma\in\R$,
    we show that the potentials $V_{s,\gamma}(x)=s^{-1}V(s^\gamma x)$ are regular for a.e. $s>0$.
    For every $s>0$, let $\mu_s$ be the measure minimizing $\cE_{V}$ in $\cM_s(\R)$ given by Theorem~\ref{teo:limitmeasure}, and let $u_s$ be given by~\eqref{eq:potentials s}.
    
    By definition, $V_{s,\gamma}$ is regular if the couple $(u_{V_{\gamma,s}},\mu_{V_{\gamma,s}})$ (given by Theorem \ref{teo:limitmeasure}) satisfies the properties of Definition~\ref{defn:regular-potential-continuous}. On the other hand, Remark~\ref{rmk:rescaling_continuous} implies that the couple $(u_{V_{\gamma,s}},\mu_{V_{\gamma,s}})$ satisfies the properties of the definition if and only if the couple $(u_s,\mu_s)$ does. So, we need to prove that $(u_s,\mu_s)$ satisfy the properties of Definition~\ref{defn:regular-potential-continuous} for a.e. $s>0$.
    But this holds thanks to Corollary~\ref{cor:generic top}, which we can apply by Proposition~\ref{prop:monotonicity_continuousCase}.

\noindent
    \textbf{Step 2: Stability of regular potentials.}
    Let $V_k\to V$ in $C^{2,\alpha}_\loc(\R)$, with $\lim_{|x|\to+\infty} \frac{V_k(x)}{\log|x|}=+\infty$ uniformly in $k$. Let $u_k=u_{V_k}$ and $u=u_V$ be defined as in \eqref{eq:potentials}. We prove that $u_k\to u$ locally uniformly and the sets $\{u_k=-V_k\}$ are equibounded. We can then apply Lemma~\ref{lem:stability singular top applied} to conclude.

    To do this, given $\delta>0$,  consider $v_\delta^\pm$ solutions of~\eqref{eq:top} with obstacle $-V \pm\delta$ and satisfying $v_\delta^\pm\sim -(1\mp\delta)\log|\cdot|$ as $|x|\to+\infty$.
    Since $\supp (\Delta v_\delta^+)$ is compact, there exists $C_0$ sufficiently large such that $v_\delta^+$ still solves~\eqref{eq:top} with obstacle $$-\tilde V_\delta=\max\big\{-V+\delta,-C_0\big(1+\log(1+|x|)\big)\big\}.$$
    Since $V_k\to V$ locally and the potentials diverge at infinity faster than a logarithm,
   for $k$ sufficiently large we have
    \[-V_k\le -\tilde V_\delta \text{ on }\R\quad \text{ and }\quad -V_k\ge -V-\delta\text{ on }\supp(\Delta v_\delta^-).\]
    We claim that $v_\delta^-\le u_k\le v_\delta^+$ for $k\gg 1$.
    Indeed, for $k\gg 1$, the functions $u_k$ and $v_\delta^+$ solve~\eqref{eq:top} in $B_R$ with ordered obstacles. Moreover, since $u_k\sim -\log R\ll v_\delta^+$ on $\partial B_R$ for $R$ sufficiently large, we can apply Remark~\ref{rmk:comparisonPrinciple-top} to get $u_k\le v_\delta^+$.
    Analogously, $v_\delta^-\le u_k$.
    
    Since $v^\pm_\delta\to u$ locally uniformly as $\delta\to0$ we find that $u_k\to u$ locally uniformly. In addition, the contact sets $\{u_k=-V_k\}$ are equibounded, as we wanted.

\noindent
    \textbf{Step 3: Higher order expansion.}
    As $\frac{d\mu_V}{dx}=-\frac1\pi\partial_2u_V$ (where we compute $\partial_2u_V(\cdot,0)$ as the limit from $\{x_2>0\}$, we simply need to show the result for $\partial_2u_V$.
    Since $u_V=-V\in C_{\rm loc}^{k+1/2+\beta}$ in the interior of $\supp\mu_V$, local $C^{k-1/2+\beta}$ regularity in the interior of $\supp\mu_V$ follows from boundary regularity for the Dirichlet problem (here $\beta\neq1/2$).

    Let us show the regularity at the boundary of $\supp\psi$.
    Up to a translation and rescaling, we can assume that $u$ solves in the unit ball
    $B_1$ the following problem:
    \[\begin{cases}
        \Delta u = 0    &\text{in }B_1\setminus \{x_2=0,x_1\le0\},\\
        u=-V       &\text{in }\{x_2=0, -1<x\le 0\},\\
        u > -V     &\text{in }\{x_2=0,0<x_1<1\}.
    \end{cases}\]
    Let $P$ be the $k$-th order ($(k+1)$-th for $\beta>1/2$) Taylor expansion of $-V$ at $0$, and let $\tilde P$ be its unique harmonic extension to $\R^2$ which is even in the $x_2$ variable.
    Define $\tilde u\coloneqq u-(-V +\tilde P-P)$. Then $\tilde u$ is even in $x_2$ and satisfies
    \[\begin{cases}
       |\Delta \tilde u|\le C|x|^{k+\beta-3/2} &\text{in }B_1\setminus\{x_2=0,x_1\le0\},\\
       \tilde u=0 &\text{in }\{x_2=0,-1<x_1\le0\}.
    \end{cases}\]
    Defining in polar coordinates
    \[
    U_0(\rho,\theta)\coloneqq \rho^{1/2}\cos\theta/2,
    \]
    we can apply Theorem~\ref{teo:boundary-harnack} to find a polynomial $Q(x_1,\rho)$ of degree $k$ such that
    \[
    \big|\tilde u-QU_0\big|\le C|x|^{k+\beta+1/2},
        \qquad \Delta (QU_0)=0,\qquad\text{and}
        \qquad Q(0)=0,
    \]
    (to deduce $Q(0)=0$, note that $|(u+V)(\cdot,0)|\le Cr^{3/2}$ by Lemma~\ref{lem:regularity-top}, which implies $|\tilde u|\le Cr^{3/2}$).
    Recalling the definition of $\tilde u$, this implies
    \[
        |u-\tilde P - U_0 Q|\le C r^{k+1/2+\beta}\quad\text{in }B_r.
    \]
    Let us set $\cC \coloneqq B_1\setminus (B_{1/4}\cup\{x_2=0,x_1\le0\})$. Given $0<r<1$ the function
    \[
        v_r\coloneqq \frac{u(r\,\cdot)-\tilde P(r\,\cdot) -U_0(r\,\cdot)Q(r\,\cdot)}{r^{k+1/2+\beta}}
    \]
    satisfies
    \[
        \Delta v_r=0\quad\text{in }\cC,\quad|v_r|\le C\text{ in }\cC,\quad\text{and}\quad\|v_r(\cdot,0)\|_{C^{k+1/2+\beta}(-1,-1/4)}\le C,
    \]
    so elliptic regularity yields
    \[
        \|v_r\|_{C^{k+1/2+\beta}}\le C\quad\text{in }B_{3/4}\cap\{x_2>0\}\setminus B_{1/2}.
    \]
    As $\partial_2\tilde P(\cdot,0)\equiv0$ (recall that $\tilde P$ is even in $x_2$) we have
    \[
        |\partial_2 u-\partial_2(U_0Q)|\le Cr^{k-1/2+\beta}\quad \text{in }B_{3r/4}\cap\{x_2=0\}\setminus B_{r/2},
    \]
    where $\partial_2$ is computed as the limit from $\{x_2>0\}$.
    By a direct computation and recalling that $Q(0)=0$, we have $\partial_2(U_0Q)(x_1,0)=(x_1)_-^{1/2}\cP$ for some polynomial $\cP(x_1)$ of degree $k-1$. Thus we can rewrite the previous inequality as
    \[
    \left|\frac{\partial_2u(x_1,0)}{(x_1)_-^{1/2}}-\cP(x_1)\right|\le Cr^{k-1+\beta}\quad\text{for }x_1\in(-3r/4,-r/2).
    \]
    This yields the desired $C^{k-1,\beta}$ regularity for $Q_V$ at the boundary of $\supp\mu_V$, from which the result follows.
\end{proof}

\section{Proof of Theorem~\ref{teo:generic_discreteCase2}}
\label{sec:proof thm2}

Given $K=\bigcup_{h=1}^N[a_h,b_h]$ a finite union of disjoint intervals define
\[
    \cM_{s,\theta}(K)\coloneqq\left\{\nu = \eta\cH^1\,:\, 0\le\eta\le\theta,\quad \nu(\R)=s,\quad\supp\nu\subseteq K\right\}.
\]
The following result is the natural generalization of Theorem~\ref{teo:teo:limitmeasure_discrete}.
\begin{teo}\label{teo:teo:limitmeasure_discrete2}
    For every $s\in(0,\theta|K|),\theta>0$ there exists a unique measure minimizing
    \[
        \inf\{\mathcal{E}_V(\nu)\,:\,\nu\in\cM_{s,\theta}(K)\}
    \]
    supported in $K$. Moreover, a measure $\mu$ is minimizing if and only if there is a constant $C$ such that the function
\begin{equation}\label{eq:potentials2-2}
    u(x)=-\log|\cdot|*\mu(x)- C,\quad x\in\R^2,
\end{equation}
    solves
\begin{equation}\label{eq:ele_discrete2}
    \begin{cases}
        \Delta u = 0      &\text{in }\R^2\setminus K,\\
        -\pi\theta\le\partial_2u\le0  &\text{in }K,\\
        \partial_2u=0     &\text{in }K\cap\{u>-V\},\\
        -\partial_2u=\pi\theta   &\text{in }K\cap\{u<-V\},\\
        \frac{u(x)}{\log|x|}\to -s   &\text{as }|x|\to+\infty.
    \end{cases}
\end{equation}
\end{teo}
The proof of the theorem above is a direct consequence of Theorem~\ref{teo:teo:limitmeasure_discrete} and the following observation that relates scalings of the potential to rescalings of the mass.

\begin{oss}\label{rmk:rescaling_discrete}
    Given a nonnegative bounded function $\psi$, let $\psi_s\coloneqq \psi(s\,\cdot)$. Then $\psi\cH^1$ minimizes $\mathcal{E}_V$ in $\cM_{s,\theta}(K)$ if and only if $\psi_s\cH^1$ minimizes
    $\mathcal{E}_{V_s}$ in $\cM_{1,\theta}(K/s)$, where $V_s = \frac{V(s\,\cdot)}{s}$.
\end{oss}

Given a potential $V$, a parameter $\theta$, and $s>0$, denote by $\psi_s$ the density minimizing $\{\mathcal{E}_V(\eta):\eta\in\cM_{s,\theta}(K)\}$, and  by $u_s$ the associated solution of \eqref{eq:ele_discrete2} given by  \eqref{eq:potentials2-2}.
\begin{prop}\label{prop:monotonicity_discrete}
    The functions $u_s$ are decreasing in $s$, namely $u_s\ge u_{s'}$ for any $s<s'$.
    More precisely, they are strictly decreasing in the following sense: for every $M>0$ there is $R_0$ such that for any $R\ge R_0$ there is $a=a(R)>0$ such that
    \begin{equation}\label{eq:strictMonotonicity-discreteCase}
    u_{s-\delta}-u_{s} > a\delta\quad\text{in }\partial B_R\cap\{|x_2|>R/2\}\end{equation}
    for every $0< s-\delta\le s\le M$ and $R\ge R_0$.
\end{prop}
The proof is based on the following Remark, which plays the same role as Remark~\ref{rmk:comparisonPrinciple-top} in the previous section.
\begin{oss}\label{rmk:comparisonPrinciple-doubleTop}
    As before, we state the comparison principle for~\eqref{eq:ele_discrete2} allowing for different obstacles.
    Let $u,v\colon B_2\subset \R^2\to\R$ be two functions, even in $x_2$, solving \eqref{eq:ele_discrete2} on possibly different unions of disjoint intervals $K_u$ and $K_v$. Assume that the obstacles of $u$ and $v$ are ordered, namely,
     $V_u\ge V_v$, where we use the convention $V_u=-\infty$ outside $K_u$ (and analogously for $v$).
    Then $(v-u)_+$ is subharmonic in $B_2$.

    Indeed, calling $\psi_u,\psi_v$ the $x_2$-derivatives of $u,v$ at $\{x_2=0\}$, respectively, we have that the distributional Laplacian of $v-u$ is equal to twice the jump of these derivatives across $\{x_2=0\}$, that is,  
    \[
        \Delta(v-u)=2(\psi_v-\psi_u)\cH^1\mres\{x_2=0\},
    \]
    where $\cH^1$ denotes the 1-dimensional Hausdorff measure. So it suffices to check that $\psi_v\ge\psi_u$ in $\{x_2=0\}\cap\{v>u\}$.
    
    Given $x\in\{v>u\}$, we consider two cases: if $v(x)\le V_v(x)$ then we have $u(x)<v(x)\le V_v(x)\le V_u(x)$, hence $\psi_u(x)=-\pi\theta\le\psi_v(x)$.
    If $v(x)>V_v(x)$ then $\psi_v(x)=0\ge\psi_u(x)$, as wanted.
\end{oss}
\begin{proof}[Proof of Proposition~\ref{prop:monotonicity_discrete}]
    The monotonicity is proven as in the proof of Proposition~\ref{prop:monotonicity_continuousCase}, using Remark~\ref{rmk:comparisonPrinciple-doubleTop} instead of Remark~\ref{rmk:comparisonPrinciple-top}.
    The proof of~\eqref{eq:strictMonotonicity-discreteCase} follows exactly that  of~\eqref{eq:strictMonotonicity-continuousCase}.
\end{proof}

\begin{proof}[Proof of Theorem~\ref{teo:generic_discreteCase2}] We split the proof in 3 steps.

\noindent\textbf{Step 1a: Regular potentials are dense -- without additional mass constraints.}
    Given $V\in C^{2,\alpha}_\loc(\R)$,
    we want to show that the potentials $V_s$ are regular for a.e. $s\in(0,\theta|K|)$, where $V_s(x)=s^{-1}V(s x)$.

    For this, we can argue as in the proof of Theorem~\ref{teo:generic_continuousCase2}, using Proposition~\ref{prop:monotonicity_discrete}, Corollary~\ref{cor:generic doubleTOP}, and Remark~\ref{rmk:rescaling_discrete} instead of Proposition~\ref{prop:monotonicity_continuousCase}, Corollary~\ref{cor:generic top},
    and Remark~\ref{rmk:rescaling_continuous}, respectively.

\noindent\textbf{Step 1b: Regular potentials are dense -- with additional mass constraints.} Given $s\in(0,\theta|K|)$ we denote by $\cN_s$ the  $(N-1)$-dimensional convex set of admissible mass constraints, i.e.,
\[
    \cN_s=\{\mathbf n=(\hat n_1,\ldots,\hat n_N)\,:\,\mathbf n\text{ satisfies }\eqref{eq:mass h}\text{ in the set }K/s\}.
\]
Given $\mathbf n\in\mathcal N_s$ we write $\mu_{s,\mathbf n}$ for the probability minimizing $\cE_{V_s}$ with the mass constraint $\mathbf n$.

Fix now $\mathbf n\in\mathcal N_1$. By Theorem~\ref{teo:teo:limitmeasure_discrete}, there are Lagrange multipliers $C_h \in \R$  such that, if we consider the modified potential
\[
    V_{\mathbf C}\coloneqq V+\sum_{h=1}^N C_h\chi_{(a_h,b_h)},\quad \mathbf C=(C_1,\ldots,C_N),
\]
then the probability measure $\mu_{1,\mathbf n}$ minimizes $\cE_{V_{\mathbf C}}$ among all probabilities without additional mass constraints.
Then, it follows from Step 1a that the probability measure  $\tilde\mu_s$ minimizing $\cE_{(V_{\mathbf C})_s}$ (without additional mass constraints) is regular for a.e. $s\in(0,\theta|K|)$. Since, by construction, $(V_{\mathbf C})_s$ equals $V_s$ up to constants in each interval, it follows that $\tilde\mu_s$ minimizes $\cE_{V_s}$ among measures with the additional mass constraint
$$\gamma_{s,\mathbf n}:=\big(\tilde\mu_s((a_1/s,b_1/s)),\ldots,\tilde\mu_s((a_N/s,b_N/s))\big) \in \mathcal N_s.$$
In particular, by the previous discussion and Step 1a, for $\mathcal H^1$-a.e. $s\in(0,\theta|K|)$ and all $\mathbf n\in\mathcal N_1$, the probability measure $\mu_{s,\gamma_{s,\mathbf n}}$ is regular.

We now note that, by continuity, given $\bar{\mathbf n}\in\mathcal N_1$ there is $\delta>0$ such that if $|1-s|+|\mathbf n-\bar{\mathbf{n}}|<\delta$ then
\[
    \gamma_{s,\mathbf n}\in\mathcal N_1\cap\mathcal N_s.
\]
Thus, by applying the argument above to the measure $\mu_{s,\gamma_{s,\mathbf n}}$ (namely, finding constants $\mathbf C=(C_1,\dots,C_N)$ so that $\mu_{s,\gamma_{s,\mathbf n}}$ minimizes $\cE_{V_s+\mathbf C}$ and considering the probability minimizing $(V_s+\mathbf C)_{1/s}$) we see that the map $\cN_1\ni\mathbf n\mapsto \gamma_{s,\mathbf n}$ is surjective onto $|\mathbf n'-\bar{\mathbf n}|<\delta$, for any $|1-s|<\delta$.
If we now show that it is also locally Lipschitz, the Area Formula will imply that for $\mathcal H^1$-a.e. $s$ and for $\cH^{N-1}$-a.e. $\mathbf n'\in\mathcal N_s$  satisfying $|1-s|+|\mathbf n'-\bar{\mathbf n}|<\delta$, the measure $\mu_{s,\mathbf n'}$ is regular, concluding the proof.

So, it is left to show that $\mathbf n\mapsto \gamma_{s,\mathbf{n}}$ is locally Lipschitz. To this aim, consider the map $ \mathbf n\mapsto \mathbf C$ defined by mapping $\mathbf n\in\mathcal N_1$ to the constants $\mathbf C=(C_1,\dots,C_N)$ given by Theorem~\ref{teo:teo:limitmeasure_discrete} and chosen such that $\sum_{h=1}^N C_h=0$.
Also, for a given $s$ close to $1$, consider the inverse map $\mathbf C\mapsto \mathbf n' \in \mathcal N_s$ that provides the masses of the minimizing measure for $(V_{\mathbf C})_s.$

Then, by construction, $\mathbf n\mapsto \gamma_{s,\mathbf{n}}$ is given by a composition of these two maps, so it suffices to prove that each of them is locally Lipschitz.
To do so, we introduce the quantity
\[
    \mathcal D^2(\eta)=-\iint_{\R^2} \log|x-y|d\eta(x)d\eta(y).
\]
When $\eta(\R)=0$ then (see \cite[Equation (36)]{borodin2017gaussian} and \cite{fourierLogEnergy})
\begin{equation}\label{eq:fourier transform log energy}
    \cD^2(\eta) = \int_0^{+\infty}t^{-1}\left|\cF_\R\eta\right|^2dt=\frac1\pi\int_{\R^2}|\cF_{\R^2}\eta|^2|\xi|^{-2}d\xi,
\end{equation}
where $\cF_{\R^k}$ is the Fourier transform in $\R^k$, and we identify a measure on $\R$ with a measure on $\R^2$ supported on $\R\times\{0\}$.
In addition, if we set
\[
    u_\eta(x)\coloneqq -\log|\cdot|*\eta(x),\quad x\in\R^2,
\]
whenever $\eta(\R)=0$ we also have
\begin{equation}\label{eq:H1 estimates}
    \cD^2(\eta)=\frac1{2\pi}\int_{\R^2}|\nabla u_\eta|^2,
\end{equation}
as can be seen using $\hat\eta(0)=0$ and \cite[Equation (5)]{fourierLogEnergy}.
We are now ready to prove the local Lipschitz regularity of the two maps described before, which will conclude the proof of this step.

\smallskip\noindent\textit{$\bullet$ $\mathbf n\mapsto\mathbf C$ is locally Lipschitz.} 
    Let $\bar {\mathbf n}\in \cN_1$.
    It follows\footnote{In the notation of \cite{borodin2017gaussian}, a ``band'' is (a connected component of) $\{0<\frac{d\mu}{dx}<\theta\}$. As $\frac{d\mu}{dx}\in C^{1/2}_{\mathrm{loc}}(\cup_h(a_h,b_h))$ by Corollary \ref{cor:optimal-regularity-doubleTop}, these are not empty in each interval, so the assumptions of \cite[Proposition 5.8]{borodin2017gaussian} are automatically satisfied for an admissible mass constraint.} from \cite[Proposition 5.8]{borodin2017gaussian} that the measures $\mu_{1,\mathbf n}$ depend continuously on $\mathbf n$, in the topology induced by $\cD$. Now \eqref{eq:H1 estimates} implies that the associated potentials given by \eqref{eq:potentials2} depend continuously in the $H^1(\R^2)$ topology, which by Corollary~\ref{cor:optimal-regularity-doubleTop} can be upgraded to $C^1_{\rm loc}$ topology, away from the endpoints $a_h,b_h$. Hence,  for each $h\in\{1,\dots,N\}$ there is an interval $I_h'\subset (a_h,b_h)$ such that $I_h'\subset \{0<\mu_{1,\mathbf n}<\theta\}$ for all $\mathbf n$ sufficiently close to $\bar {\mathbf n}$. Since the constant in \cite[Proposition 5.8]{borodin2017gaussian} depends only on these sets, it can be chosen locally independent from $\mathbf n$, that is,
    \[
        \cD(\mu_{1,\mathbf n_1}-\mu_{1,\mathbf n_2})\le C|\mathbf n_1-\mathbf n_2|\qquad \text{for }|\mathbf n_1-\bar {\mathbf n}|+|\mathbf n_2-\bar {\mathbf n}|<\delta
    \]
    In particular, if $u_1,u_2$ are the potentials given by \eqref{eq:potentials2}, it follows from \eqref{eq:H1 estimates} that
    \[
        \|\nabla(u_1-u_2)\|_{L^2(\R^2)}\le C|\mathbf n_1-\mathbf n_2|\qquad\text{and}\qquad u_1-u_2 = (\mathbf C_1)_h-(\mathbf C_2)_h\quad\text{in }I_h'.
    \]
   Hence, using Poincaré and trace inequalities, we conclude that, for some $\mathbf A=(a,\dots,a)\in\R^N$,
    \[
        |\mathbf C_1-\mathbf C_2-\mathbf A|\le C|\mathbf n_1-\mathbf n_2|.
    \]
    Since the vectors $\mathbf C_i$ have zero average, the claim follows.

    \smallskip\noindent$\bullet$ \textit{$\mathbf C\mapsto \mathbf n'$ is locally Lipschitz.}
    Let $s>0$ be close to $1$, and  let $\mathbf C_1,\mathbf C_2$ be given. For $i=1,2$ we write $V_i=V_s+\sum_h (\mathbf C_i)_h\chi_{(a_h/s,b_h/s)}$, $\mu_i=\mu_{V_i},$ and we denote by $u_i$ the potential associated to $\mu_i$ by \eqref{eq:potentials2}.
    For all probabilities $\nu\in\cM_{1,\theta}(K/s)$ we can write
    \[
        \cE_{V_2}(\nu) = \cE_{V_2}(\mu_1)+\cD^2(\nu-\mu_1)+2\scalar{V_1+u_1,\nu-\mu_1}+2\scalar{V_2-V_1,\nu-\mu_1},
    \]
    where $\scalar{\cdot,\cdot}$ denotes the duality between measures and integrable functions.
    As $\cE_{V_2}(\mu_2)\le \cE_{V_2}(\mu_1)$ (by minimality of $\mu_2$), we find
    \[
        \cD^2(\mu_2-\mu_1)\le -2\scalar{V_1+u_1,\mu_2-\mu_1}+2\scalar{V_1-V_2,\mu_2-\mu_1}.
    \]
    \Cref{eq:ele_discrete} gives $\scalar{V_1+u_1,\mu_2-\mu_1}\ge0$. To bound the second term, for every interval $(a_h/s,b_h/s)$ fix a function $0\le\varphi_h\in C^\infty_c(\R^2)$ such that $\varphi_h\equiv1$ on $(a_h/s,b_h/s)$ and $\varphi_h\equiv 0$ in $K/s\setminus (a_h/s,b_h/s)$, and define
    \[
        W\coloneqq \sum_h\varphi_h(\mathbf{C}_{1,h}-\mathbf C_{2,h}).
    \]
    Since $\mu_2-\mu_1$ is supported on $K/s$ while $V_1-V_2 = \mathbf{C}_{1,h}-\mathbf C_{2,h}$ on $(a_h,b_h)/s$, using Plancherel and \eqref{eq:fourier transform log energy} we find
    \[\begin{split}
        |\scalar{V_1-V_2,\mu_2-\mu_1}| &= |\scalar{W,\mu_2-\mu_1}| = \left|\int_\R \cF_\R W \cF_\R(\mu_1-\mu_2)\right|\\
            &\le \left(\int_\R |t||\cF_\R W|^2\right)^{1/2}\left(\int_\R|t|^{-1}|\cF_\R(\mu_1-\mu_2)|^2\right)^{1/2}\\
            &\le C|\mathbf{C}_1-\mathbf{C}_2|\cD(\mu_2-\mu_1).
    \end{split}\]
    Thus
    \[
        \cD^2(\mu_2-\mu_1)\le C|\mathbf C_1-\mathbf C_2|\cD(\mu_2-\mu_1),
    \]
    so thanks to \eqref{eq:H1 estimates}
    \[
        \|\nabla(u_1-u_2)\|_{L^2(\R^2)}\le C|\mathbf C_1-\mathbf C_2|.
    \]
    If $\varphi_h$ are as above we compute
    \[
        |(\mathbf n_1')_h-(\mathbf n_2')_h|=\frac1{2\pi}\left|\int_{\R^2}\varphi_h\Delta(u_1-u_2)\right| = \frac1{2\pi}\left|\int_{\R^2}\nabla\varphi_h\cdot\nabla(u_1-u_2)\right|\le C|\mathbf C_1-\mathbf C_2|,
    \]
    as we wanted.

\noindent\textbf{Step 2: Stability of regular potentials.}
    Suppose $V_k\to V$ in the topology $(\tilde X,\tilde\rho)$ of Section \ref{sec:min obst} with $V$ regular. Let $u_k,u$ be the functions given by Theorem~\ref{teo:teo:limitmeasure_discrete}. We prove that $u_{k}\to u$ locally uniformly in $\R^2$, since this allows to conclude by Lemma~\ref{lem:stability singular doubleTOP applied}.
    
    Given $\delta>0$,  consider $v_\delta^\pm$ solutions of~\eqref{eq:ele_discrete} with obstacle $-V^\pm_\delta(x)=-V((1\mp\delta)x)\pm\eps(\delta)$
    and satisfying\footnote{The existence of these functions can be proved, for instance, using Theorem~\ref{teo:teo:limitmeasure_discrete2} (namely, minimizing $\mathcal E_{V^\pm_\delta}$ with an upper bound $\theta$ on the density and suitable mass constraints).}
$$v_\delta^\pm\sim-(1\mp\delta)\log|\cdot|\qquad\text{as $|x|\to+\infty$,}
    $$
    where $\eps(\delta)\to0$ as $\delta\to0$ and is chosen so that $$- V_\delta^+\ge -V+\delta \geq -V-\delta \geq - V_\delta^-.$$ 
    Since $-V^-_\delta\le -V_k\le -V^+_\delta$ for $k\gg 1$,  Remark~\ref{rmk:comparisonPrinciple-doubleTop} implies that $v_\delta^-\le u_k\le v_\delta^+$. Since $v_\delta^\pm\to u$ as $\delta\to0$, the claim follows.

\noindent\textbf{Step 3: Higher order expansion.}
    The argument is the same as in Step 3 in the proof of Theorem~\ref{teo:generic_continuousCase2}.
\end{proof}

\appendix

\section{Results from Thin Obstacle Problem}
\label{sect:TOP}

Here we collect useful results from the regularity of the thin obstacle problem, referring the interested reader to \cite[Chapter 9]{Petrosyan2012} and the recent survey \cite{FernandezReal2022} for more details.
To keep the notation consistent with the rest of the paper, we restrict ourselves to the two-dimensional setting.

Recall that, given a function $\varphi\colon[-1,1]\to\R$, a (two-dimensional) solution of the thin obstacle problem with obstacle $\varphi$ is a function $u\colon B_1\subset \R^2\to\R$ satisfying
\begin{equation}\label{eq:top}
    \begin{cases}
        \Delta u =0                 &\text{in }B_1\setminus\{x_2=0\},\\
        \Delta u\le0         &\text{in }\{x_2=0\}\cap B_1\\
        \partial_{x_2}u=0           &\text{in }\{x_2=0\}\cap\{u>\varphi\}\cap B_1,\\
        u\ge\varphi                 &\text{in }\{x_2=0\}\cap B_1.
    \end{cases}
\end{equation}
We will denote the contact set of $u$ by $\Lambda(u)\coloneqq\{u=\varphi\}\cap\{x_2=0\}$ and the free boundary by $\Gamma(u)=\partial_\R\Lambda(u)$ (i.e., the topological boundary of $\Lambda(u)$ as a subset of $\{x_2=0\}\simeq \R$).
A point $x_0\in\Gamma(u)$ is regular ($x_0\in\Reg(u)$) provided there exists $c>0$ such that
\[|u(x_0+r\cdot)-\varphi(x_0+r\cdot)|\ge cr^{3/2}\quad\text{in }B_1, \quad \forall\,\,r<1,\]
and it is singular ($x_0\in\Sing(u)$) if there exists $C>0$ such that 
\begin{equation}\label{eq:sing-points}|u(x_0+r\cdot)-\varphi(x_0+r\cdot)|\le Cr^2\quad\text{in }B_1, \quad \forall\,\,r<1,\end{equation}
where we mean $\varphi(x_1,x_2)=\varphi(x_1)$.

\begin{defn}\label{defn:regular-solution-top}
    A solution $u$ of~\eqref{eq:top} is \emph{regular} if there is no point in $\Gamma(u)$ where~\eqref{eq:sing-points} holds.
\end{defn}

We recall that solutions to the thin obstacle problem are $C^{1,1/2}$ (see~\cite{AthCaff04,ruland2017optimal}), and that regular free boundaries are smooth.
\begin{lem}\label{lem:regularity-top}
    Let $u$ solve~\eqref{eq:top} in $B_1\subset \R^2$ with obstacle $\varphi\in C^{1,\alpha}([-1,1])$ for some $\alpha>\frac12$. Then $u(\cdot,0)\in C^{1,1/2}_{\rm loc}((-1,1))$. In particular, there is $C>0$ depending on $\|\varphi\|_{C^{1,1}}$ such that $\|u(\cdot,0)\|_{C^{1,1/2}([-1/2,1/2])}\le C\|u\|_{L^2(B_1)}$.
\end{lem}
\begin{lem}\label{lem:fb regularity}
    Let $u$ solve \eqref{eq:top} in $B_1\subset\R^2$ with obstacle $\varphi\in C^{1,1}$, and assume that $u$ is regular in the sense of Definition \ref{defn:regular-solution-top}. Then $\Gamma(u)$ is discrete and if $0\in \Gamma(u)$ then, up to a rotation,
    \[
        u(r\cdot)-\varphi(r\cdot) = cr^{3/2}\operatorname{Re}z^{3/2} + o(r^{3/2})\quad\text{in }(-1,1)
    \]
    for some $c>0$. In particular, $\supp(-\Delta u) = \{u=\varphi\}$.
\end{lem}
\begin{proof}
    See \cite[Theorem 5.3 and Corollary 5.7]{FernandezReal2022}.
\end{proof}

While singular points may exist, it has been recently proven that, in low dimensions, the singular set is generically empty.
More precisely, consider  a continuous $1$-parameter family of solutions $\{u_t\}_{t\in[-1,1]}$ to~\eqref{eq:top} in $B_1$, strictly increasing in the following sense:
there exists $\eta>0$ such that 
\begin{equation}\label{eq:strictMonotonicity_HP_TOP}\begin{cases}u_{t+\eps}\ge u_t    &\text{in }\partial B_1\\
    u_{t+\eps}\ge u_t + \eta\eps       &\text{in }\partial B_1\cap\{|x_2|>\frac12\}
\end{cases}\end{equation}
for every $\eps>0$.

\begin{teo}[Generic regularity for the thin obstacle problem, see~\cite{FernandezRealRosOton,FERNANDEZREAL2023109323}]\label{teo:genericTOP}
    Let $\varphi\in C^{2,\alpha}([-1,1])$ for some $0<\alpha\le1$ and let $u_t:B_1\to \R$ be a family of solutions of~\eqref{eq:top} satisfying~\eqref{eq:strictMonotonicity_HP_TOP}.
    Then, for a.e.  $t\in[-1,1]$, $u_t$ is regular in the sense of Definition \ref{defn:regular-solution-top}.
\end{teo}

We also recall two useful properties: limits of singular points are singular, and solutions with smoother obstacles are more regular.
\begin{lem}[Convergence of singular points, see~\cite{FernandezRealRosOton}]\label{lem:stability singular top}
    Let $u_k,u$ solve~\eqref{eq:top} in $B_1$ with obstacles $\varphi_k,\varphi\in C^{1,1}([-1,1])$.
    Assume that
    \begin{itemize}
        \item[\textit{i)}] $u_k\to u, \varphi_k\to \varphi$ uniformly, and $\sup_{k}\|\varphi_k\|_{C^{1,1}([-1,1])}<+\infty$;
        \item[\textit{ii)}] there are points $x_k\in\Sing(u_k)$ with $x_k\to x\in B_{1/2}$.
    \end{itemize}
    Then $x\in\Sing(u)$.
\end{lem}

\begin{teo}[Higher order boundary Harnack, {see~\cite[Theorem 2.3]{de_silva_savin_2016}}]\label{teo:boundary-harnack}
    Let $u$ be even in $x_2$ and solve
    \[
    \begin{cases}
        |\Delta u|\le C |x|^{k+\alpha-3/2} &\text{in }B_1\setminus \{x_2=0,x_1\le0\},\\
        u=0&\text{in }\{x_2=0,x_1\le0\}.
    \end{cases}
    \]
    for some $k\ge1$ and $\alpha>0$.

    Then there exist a polynomial $Q = Q(x_1,\rho)$ of degree $k$ and a constant $C'>0$ such that
    \[
    \sup_{B_r}|u-QU_0|\le C'r^{k+\alpha+1/2}\quad \text{ for all }0<r<1,
    \]
    where $U_0(\rho,\theta)=\rho^{1/2}\cos\theta/2$.
\end{teo}

\subsection{Applications to Section \ref{sec:min obst}}
We note explicitly the following consequences of these results in the context of Theorem~\ref{teo:generic_continuousCase2}, based on the following equivalence of Definitions \ref{defn:regular-solution-top} and \ref{defn:regular-potential-continuous}.
\begin{lem}\label{lem:equivalence regular definitions}
    Suppose that a positive measure $\mu$ is such that
    \[
        u_\mu(x) \coloneqq -\log|\cdot|*\mu(x) + C,\quad x\in\R^2
    \]
    solves \eqref{eq:ele} for some constant $C\in\R$, and some potential $V\colon\R\to\R$. Then the couple $(u_\mu,\mu)$ satisfies the properties of Definition~\ref{defn:regular-potential-continuous}  if and only if $u_\mu$ is a regular solution in the sense of Definition~\ref{defn:regular-solution-top} and $\{u_\mu = -V\}$ is compact. 
\end{lem}
\begin{proof}
    On one hand, suppose the couple $(u_\mu,\mu)$ satisfies Definition~\ref{defn:regular-potential-continuous} and let $x_0\in\partial\{u_\mu=-V\}$. As $\frac{d\mu}{dx} = -\frac1\pi\partial_2u_\mu(\cdot,0)$ we deduce $\|\partial_2u_\mu(x_0+r\cdot)\|_{L^2(B_1)}\ge c r^{1/2}$ for $r$ small, so \eqref{eq:sing-points} cannot hold.

    Viceversa, if $u_\mu$ is regular in the sense of Definition~\ref{defn:regular-solution-top} then Lemma \ref{lem:fb regularity} implies that the set $\{u_\mu = - V\} = \supp\mu$ is a finite union of intervals. In addition, if $0\in\partial\supp\mu$ then up to a rotation we can write
    \[
        u(r\cdot)+V(r\cdot) = cr^{3/2}\operatorname{Re}z^{3/2} + o(r^{3/2})\quad\text{in }(-1,1)
    \]
    for some $c>0$. By harmonic estimates this implies
    \[
        \partial_2u(r\cdot) = cr^{1/2}\partial_2\operatorname{Re}z^{3/2} + o(r^{1/2})\quad\text{in }(-1,1)
    \]
    which, using $c>0$, yields \eqref{eq:form-regular-potential}.
\end{proof}
\begin{lem}\label{lem:stability singular top applied}
    Suppose that $u_k,u$ are of the form
    \[
        u_k = -\log|\cdot|*\mu_k + C_k,\quad u = -\log|\cdot|*\mu + C
    \]
    for some nonnegative measures $\mu_k,\mu$ and constants $C_k,C$. Assume in addition that they solve \eqref{eq:ele} with potentials $V_k,V$ and:
    \begin{itemize}
        \item[\textit{i)}] $u_k\to u, V_k\to V$ locally uniformly, and $\sup_{k}\|V_k\|_{C^{1,1}_\loc(\R)}<+\infty$;
        \item[\textit{ii)}] $(u,\mu)$ satisfies the properties of Definition~\ref{defn:regular-potential-continuous};
        \item[\textit{iii)}] the set $\cup_k \{u_k=-V_k\}$ is bounded.
    \end{itemize}
    Then $(u_k,\mu_k)$ satisfies Definition~\ref{defn:regular-potential-continuous} for $k$ large enough.
\end{lem}
\begin{proof}
    It follows from Lemmas~\ref{lem:stability singular top} and \ref{lem:equivalence regular definitions}.
\end{proof}

\begin{cor}\label{cor:generic top}
    Let $V\in C^{2,\alpha}(\R)$ satisfy $\lim_{|x|\to+\infty}\frac{v(x)}{\log|x|}=+\infty$, and let $(u_t)_{t>0}$ be a continuous family of functions of the form
    \[
        u_t = -\log|\cdot|*\mu_t + C_t
    \]
    for some nonnegative measures $\mu_t$ and constants $C_t$,
    solving \eqref{eq:ele} and strictly monotone in the sense of Proposition~\ref{prop:monotonicity_continuousCase}. Suppose in addition that, for each $t>0$, $\limsup_{|x|\to+\infty} \frac{u_t(x)}{-\log|x|}<+\infty$. Then for a.e. $t>0$ the couple $(u_t,\mu_t)$ satisfies the properties of Definition~\ref{defn:regular-potential-continuous}.
\end{cor}
\begin{proof}
    As $-V(x)\ll -\log|x|\lesssim u_t(x)$ as $|x|\to+\infty$, for each $t$ the contact set $\{u_t = -V\}$ is compact. So, the result follows from Theorem~\ref{teo:genericTOP} and Lemma~\ref{lem:equivalence regular definitions}.
\end{proof}

\section{Phase Separation for a Two-Phase Problem}
\label{sec:phase separation}

As mentioned in the introduction, the goal of this section is to establish a crucial result that will later be used in Appendix~\ref{sec:double top} to prove generic regularity for \eqref{eq:ele_discrete}: in the two-phase thin obstacle problem, phases with opposite signs cannot touch. This phase separation was already stated in \cite[Proposition 4.1]{Caffarelli2020}, but the proof given there contains a gap.\footnote{The proof of \cite[Proposition 4.1]{Caffarelli2020} relies essentially on \cite[Lemmas 3.5 and 3.6]{Caffarelli2020}, but their argument is incorrect. In particular, \cite[Equation (3.16)]{Caffarelli2020} is wrong, since the function considered there is not super-harmonic and therefore the maximum principle does not apply.} For this reason, in this appendix we work in a more general framework that is both suitable for our purposes and provides a new proof of \cite[Proposition 4.1]{Caffarelli2020}.

More precisely, we consider the problem in $d+1$ dimensions, with $d \geq 1$ an integer, and study the slightly more general formulation \eqref{eq:double top} below. A related phase separation result was also established in \cite{Allen15}, where the two-phase obstacle problem for the fractional Laplacian was studied. However, that result is not sufficient for either \eqref{eq:ele_discrete} or the setting of \cite{Caffarelli2020}. More specifically, \cite[Theorem 9.1]{Allen15} coincides (in the case $a=0$) with Theorem~\ref{teo:phase separation}, but only in the special case $f \equiv \theta$ and $\varphi \equiv 0$. While their argument could probably be adapted also to our setting, we present here a new independent proof, which we believe is of independent interest.

\smallskip 

We denote points in $\R^{d+1}$ with $(x,y)\in\R^d\times\R$.
Given a set $S\subset\R^{d+1}$ we write
\[
    S^+=S\cap\{y>0\}\quad\text{and}\quad S'=S\cap\{y=0\}.
\]
We fix functions $\varphi\colon B_1'\to\R$ of class $C^{1,1}$ and $f\colon [0,+\infty)\to(0,+\infty)$ satisfying:
\begin{enumerate}[label=\alph*)]
    \item\label{item:assumption a} $f(s)\le f(0)\eqcolon\theta$ for all $s\in [0,+\infty)$;
    \item\label{item:assumption b} $\|f'\|_{L^\infty(0,+\infty)}<+\infty$;
\end{enumerate}
and we study solutions $u\in H^1(B_1^+)$ of
\begin{equation}\label{eq:double top}\begin{cases}
    \Delta u=0  &\text{in }B_1^+,\\
    |\partial_yu|\le\theta\coloneqq f(0)  &\text{in }\{y=0\}\cap B_1,\\
    \partial_yu=f(|u-\varphi|)\sign(u-\varphi)  &\text{in }\{y=0,u\neq\varphi\}.
\end{cases}\end{equation}
When $d=1$ and $f(s)\equiv\theta$, this problem corresponds to \eqref{eq:ele_discrete} by replacing $u$ with $\frac2\pi u+\theta|y|$.
It has also been studied in \cite{Caffarelli2020} in the case $\varphi\equiv0$. In \cite{Allen15} it has been studied in the case $\varphi\equiv0$ and $f\equiv\theta$ but with the fractional laplacian $(-\Delta)^s$ for all $s\in(0,1)$.

The main challenge in the study of \eqref{eq:double top} is to exclude the existence of ``two phase points'', namely points in the set
\[
    \Sigma_{\mathrm{two\,phase}}\coloneqq\partial\{u(\cdot,0)>\varphi\}\cap\partial\{u(\cdot,0)<\varphi\},
\]
as they correspond to jumps in the normal derivative. The main result of this section is a positive answer to this question.
\begin{teo}\label{teo:phase separation}
    Let $u\in H^1(B_1^+)$ solve \eqref{eq:double top} with $\varphi\in C^{1,1}$ and $f$ satisfying assumptions \ref{item:assumption a} and \ref{item:assumption b}. Then $\Sigma_{\mathrm{two\,phase}}=\emptyset$. More precisely, there is $r>0$ depending only on $d,\theta,\|u\|_{L^2(B_1^+)},\|f'\|_\infty$, and $\|\varphi\|_{C^{1,1}}$ such that, for all $x\in B_{1/2}'$, either $u(\cdot,0)\ge\varphi$ or $u(\cdot,0)\le\varphi$ in $B_r(x)'$.
\end{teo}
Once phase separation is established, it is possible to adapt techniques of the one-phase thin obstacle problem to show optimal regularity (namely, $C^{1,1/2}$) of the solutions as well as regularity of the free boundary (that is, the set $\partial\{u(\cdot,0)\neq\varphi\}$) close to non-degenerate points.
We refer to \cite{Caffarelli2020} and references therein for more details on the connection between phase separation and regularity in the case $\varphi\equiv0$, as well as other two-phase problems in the literature where phase separation holds.
In Appendix~\ref{sec:double top} we prove optimal regularity of the solutions and generic regularity of the free boundary when $f\equiv\theta$ and $d=1$.

The strategy to prove Theorem~\ref{teo:phase separation} is as follows: we use a Bernstein-type technique to show local Lipschitz regularity of the solutions and we exploit the ACF monotonicity formula to prove that Lipschitz solutions meet the obstacle tangentially at contact points (that is, in the set $\{u(\cdot,0)=\varphi\}$). We then conclude with a barrier argument borrowed from~\cite{Caffarelli2020}.

\subsection{Local Lipschitz Regularity}
Note that \eqref{eq:double top} implies
\begin{equation}\label{eq:consequence pde}
    |\partial_yu(x,0)|\le f(|u(x,0)-\varphi(x)|)\quad \forall\, x\in B_1'.
\end{equation}
Here $\nabla_xu$ (resp. $\Delta_xu$) denotes the gradient (resp. the laplacian) of $u$ with respect to the $x$ variable, while $\nabla u$ (resp. $\Delta u$) denotes the full gradient (resp. the full laplacian) of $u$.

\begin{lem}\label{lem:tangential Lip}
    There is $C>0$, depending only on $d$ and $\|f'\|_\infty$, such that if $u\in H^1(B_1^+)$ solves $\eqref{eq:double top}$ then
    \[
        |\nabla_xu|\le C\big(\|\nabla_xu\|_{L^2(B_1^+)}+\|\nabla\varphi\|_\infty\big)\quad\text{in }B_{1/2}^+.
    \]
\end{lem}
\begin{proof}
    Given $h\in\R^d$ with $|h|<1/3$, we set
    \[
        v=\frac{u-u(\cdot+h,\cdot)}{|h|}-\|\nabla\varphi\|_\infty,
    \]
    where $u$ is extended by even reflection to $\{y<0\}$.
    We claim that
    \begin{equation}\label{eq:pde derivatives}
        \Delta v_+\ge -2\|f'\|_\infty(v_++2\|\nabla\varphi\|_\infty)\cH^d\mres\{y=0\}\quad \text{in }B_{2/3}.
    \end{equation}
    Since
    \[
        \Delta (u-u(\cdot+h,\cdot))=2(\partial_yu-\partial_yu(\cdot+h,0))\cH^d\mres\{y=0\}\quad \text{in }B_{2/3},
    \]
    it is sufficient to check
    \[
        \partial_yu(\cdot,0)-\partial_yu(\cdot+h,0)\ge -\|f'\|_\infty|h|(v+2\|\nabla\varphi\|_\infty)\quad\text{in }\{y=0,v>0\}.
    \]
    Given $x\in B_1'$ such that $u(x,0)>u(x+h,0)+\|\nabla\varphi\|_\infty|h|$ we split the proof in three cases:

    \noindent $\bullet$ if $u(x+h,0)\ge\varphi(x+h)$ then $u(x,0)>\varphi(x+h)+\|\nabla\varphi\|_\infty|h|\ge\varphi(x)$. Thus on one hand \eqref{eq:double top} gives $\partial_yu(x,0)=f(u(x,0)-\varphi(x))$; on the other hand \eqref{eq:consequence pde} yields $-\partial_yu(x+h,0)\ge -f(u(x+h,0)-\varphi(x+h))$. So
    \[\begin{split}
        \partial_yu(x,0)-\partial_yu(x+h,0)&\ge f(u(x,0)-\varphi(x))-f(u(x+h,0)-\varphi(x+h))\\
        &\ge-\|f'\|_\infty|u(x,0)-\varphi(x)-u(x+h,0)+\varphi(x+h)|\\
        &\ge-\|f'\|_\infty|h|(v+2\|\nabla\varphi\|_\infty);
    \end{split}\]
    
    \noindent$\bullet$ if $u(x+h,0)<\varphi(x+h)$ and $u(x,0)>\varphi(x)$, then
    \[
        \partial_yu(x,0)-\partial_yu(x+h,0)=f(u(x,0)-\varphi(x))+f(-u(x+h,0)+\varphi(x+h))\ge0;
    \]

    \noindent$\bullet$ if $u(x+h,0)<\varphi(x+h)$ and $u(x,0)\le\varphi(x)$, then
    \[\begin{split}
        \partial_yu(x,0)-\partial_yu(x+h,0)&\ge -f(-u(x,0)+\varphi(x))+f(-u(x+h,0)+\varphi(x+h))\\
        &\ge -\|f'\|_\infty|\varphi(x)-u(x,0)-\varphi(x+h)+u(x+h,0)|\\
        &\ge -\|f'\|_\infty|h|(v+2\|\nabla\varphi\|_\infty).
    \end{split}\]
    Thus, \eqref{eq:pde derivatives} holds.

    It follows that the function
    \[
        w=(1+\|f'\|_\infty|y|)(v_+ + 2\|f'\|_\infty\|\nabla\varphi\|_\infty|y|)
    \]
    satisfies
    \[
    \Delta w -\frac{2\|f'\|_\infty\sign(y)}{1+\|f'\|_\infty|y|}w_y+\frac{2\|f'\|_\infty^2}{(1+\|f'\|_\infty|y|)^2}w\ge0\quad\text{in }B_{2/3},
    \]
    hence we can apply \cite[Theorem 8.17]{GilbargTrudinger} to find
    \[
        \sup_{B_{1/2}}w\le C\|w\|_{L^2(B_1)}\le C(\|\nabla_xu\|_{L^2(B_1)}+\|\nabla\varphi\|_\infty).
    \]
    In particular
    \[
        \frac{u(x,y)-u(x+h,y)}{|h|}-\|\nabla\varphi\|_\infty\le C(\|\nabla_xu\|_{L^2(B_1)}+\|\nabla\varphi\|_\infty)\quad \text{for }(x,y)\in B_{1/2}.
    \]
    The same argument applied to $-u$ in place of $u$ yields
    \[
        \frac{u(x+h,y)-u(x,y)}{|h|}-\|\nabla\varphi\|_\infty\le C(\|\nabla_xu\|_{L^2(B_1)}+\|\nabla\varphi\|_\infty)\quad \text{for }(x,y)\in B_{1/2},
    \]
    and the result follows.
\end{proof}

\begin{lem}\label{lem:normal Lip}
    If $u\in H^1(B_1^+)$ solves \eqref{eq:double top} then
    \[
        |\partial_yu|\le\theta+Cy\quad\text{in }B_{1/2}^+,
    \]
    for some $C>0$ depending only on $d$ and $\|\nabla u\|_{L^2(B_1^+)}$.
\end{lem}
\begin{proof}
    We set $v(x,y)=u(x,y)-\varphi(x)$, so that $v\in H^1(B_1^+)$ satisfies
    \[\begin{cases}
        \Delta v =-\Delta\varphi(x)\eqcolon g(x)    &\text{in }B_1^+,\\
        \partial_y v=\sign(v)f(|v|) &\text{in }\{y=0,v\neq0\},\\
        |\partial_yv|\le\theta  &\text{in }\{y=0\},
    \end{cases}
    \]
    with $g\in L^\infty(B_1)$.
    Given $\delta>0$ small, let $h\in\R^d$ with $|h|<\delta$.
    Given a function $\psi\colon B_1\to\R$ we define $\delta_h\psi\colon B_{1-\delta}\to\R$ by $\delta_h\psi\coloneqq\frac1{|h|}(\psi(x+h,y)-\psi(x,y))$. Extending $v$ to $B_1$ by even reflection in $y$ we define
    \[
        w\coloneqq \delta_hv\quad \text{in } B_{1-\delta}.
    \]
    We note that
    \begin{equation}\label{eq:pde derivatives caccioppoli}
        w\Delta w\ge w\,\delta_hg-2\|f'\|_\infty w^2\cH^d\mres\{y=0\}\quad\text{in }B_{1-\delta},
    \end{equation}
    as can be shown arguing as in the proof of \eqref{eq:pde derivatives}.
    
    We now fix a smooth cutoff $\eta\in C^\infty_c(B_{1-\delta})$ with $0\le \eta\le1$ and $\eta\equiv1$ in $B_{1-2\delta}$.
    Using \eqref{eq:pde derivatives caccioppoli}, the identity $\Delta (w^2)/2=|\nabla w|^2+w\Delta w,$ and $\Delta |y| = 2\cH^d\mres\{y=0\}$, we compute
    \[\begin{split}
        \int_{B_1}\eta|\nabla w|^2&\le \int_{B_1}\eta\Delta(w^2/2)-\eta w\,\delta_hg+\|f'\|_\infty\eta w^2\Delta(|y|)\\
        &=\int_{B_1}[\Delta(\eta/2)-\|f'\|_\infty\nabla(\eta)\cdot\nabla(|y|)]w^2-2\|f'\|_\infty\eta w\nabla w\cdot\nabla(|y|)+g\,\delta_{-h}(\eta w)\\
        &\le\int_{B_1} Cw^2 + \tfrac12\eta|\nabla w|^2+C\|g\|_2^2.
    \end{split}\]
Thus, recalling that $w=\delta_h v$,
    \[
        \int_{B_{1-2\delta}}|\delta_h\nabla v|^2\le C\int_{B_1}|\nabla v|^2+C\|g\|_2^2.
    \]
    Using this and $\partial_{yy}v = g-\Delta_x v$, we deduce that $v\in H^2(B_r^+)$ for all $r\in(0,1)$. In particular $\partial_yu=\partial_yv\in H^1(B_r^+)$ for all $r\in(0,1)$.
    
    It follows that if we extend $u$ by even reflection to $\{y<0\}$, the function $|\partial_yu|$ belongs to $H^1_\loc(B_1)$, satisfies $|\partial_yu|\le\theta$ in $\{y=0\}$, and is subharmonic in $B_1\setminus\{y=0\}$, therefore $\max\{\theta,|\partial_yu|\}$ is subharmonic in $B_1$.
    In particular, $|\partial_yu|\le M$ in $B_{2/3}^+$ for some $M$ depending only on $d,\theta$ and $\|\nabla u\|_{L^2(B_1^+)}$. As $|\partial_yu|\le\theta$ in $B_1'$, denoting by $\psi$ the solution of
    \[\begin{cases}
        \Delta \psi=0  &\text{in }B_{2/3}^+,\\
        \psi=\theta    &\text{in }B_{2/3}',\\
        \psi=M         &\text{in }\partial B_{2/3},
    \end{cases}\]
    we deduce from elliptic regularity and the maximum principle that
    \[
        |\partial_yu|\le \psi\le \theta+Cy\quad\text{in }B_{1/2}^+
    \]
    for some $C>0$ depending only on $d$ and $M$, as we wanted.
\end{proof}

\begin{lem}\label{lem:lip double top}
    There exists $C>0$, depending only on $d,\theta,\|u\|_{L^2(B_1^+)},\|f'\|_\infty$, and $\|\varphi\|_{C^1}$, such that if $u\in H^1(B_1^+)$ solves \eqref{eq:double top} then
    \[
        |\nabla u|\le C\quad \text{in }B_{1/2}^+.
    \]
\end{lem}
\begin{proof}
    By Lemmas~\ref{lem:normal Lip} and \ref{lem:tangential Lip} together with a covering argument, it suffices to bound $\|\nabla u\|_{L^2(B_{2/3}^+)}$. Fix a smooth cutoff function $\eta\in C_c^\infty(B_1)$ such that $0\le \eta$ and $\eta\equiv1$ in $B_{2/3}$.
    As $u\partial_yu \ge -\theta |\varphi|$ in $B_1'$, we compute
    \[\begin{split}
        \int_{B_1^+}\eta^2|\nabla u|^2 &= -\int_{B_1'}\eta^2u\partial_yu-2\int_{B_1^+}\eta u\nabla\eta\cdot\nabla u\\
        &\le 2\int_{B_1^+}|\nabla\eta|^2u^2+\frac12\int_{B_1^+}\eta^2|\nabla u|^2 + C\|\varphi\|_{L^\infty}
    \end{split}\]
    which yields
    \[
        \|\nabla u\|^2_{L^2(B_{2/3}^+)}\le C\|u\|^2_{L^2(B_1^+)}+C\|\varphi\|_{L^\infty},
    \]
    so the result follows.
\end{proof}

\subsection{Lipschitz Global Solutions}
We now consider Lipschitz functions $u\colon \R^{d+1}_+\to\R$ satisfying
\begin{equation}\label{eq:double top constant coefficients}\begin{cases}
    \Delta u=0  &\text{in }\R^d\times(0,+\infty),\\
    |\partial_yu|\le \theta&\text{in }\R^d\times[0,+\infty),\\
    \partial_y u(\cdot,0)=\theta \sign(u)   &\text{in }\{u(\cdot,0)\neq0\},
\end{cases}\end{equation}
which will naturally arise from blow-ups of solutions of \eqref{eq:double top}. The goal of this section is to show that there are no nontrivial blow-ups.
\begin{prop}[Classification of Lipschitz global solutions]\label{prop:global solutions}
    Let $u$ be a Lipschitz global solution of \eqref{eq:double top constant coefficients} with $u(0)=0$.
    Then $u(\cdot,0)\equiv0$.
\end{prop}
This result is~\cite[Lemma 8.3, $a=0$]{Allen15}. Here we give a different proof.
We first recall the same result for solutions of the thin obstacle problem (cf. Appendix~\ref{sect:TOP}).
\begin{lem}\label{lem:global lipschitz Signorini}
    Let $u\colon \R^{d+1}\to \R$ be a Lipschitz global solution of the thin obstacle problem \eqref{eq:top} with $u(0)=0$. Then $u(\cdot,0)\equiv0$.
\end{lem}
\begin{proof}
    By regularity for the thin obstacle problem (see e.g. \cite[Theorem 9.13]{Petrosyan2012})
    \[
        [u(\cdot,0)]_{C^{1,1/2}(B_{1/2})}\le C\|\nabla u\|_{L^{\infty}(B_1)}.
    \]
    Applying this to
    \[
        u_R(x,y)= \tfrac1Ru(Rx,Ry)
    \]
    and letting $R\to+\infty$ yields
    \[
        [u(\cdot,0)]_{C^{1,1/2}(B_{R/2})} \le  C{R^{-1/2}}\|\nabla u\|_{L^{\infty}(\R^{d+1})}\to0,
    \]
    which implies the claim.
\end{proof}

We will also make use of the ACF monotonicity formula, that we state in our setup.
\begin{lem}[{\cite[Theorems 2.4 and 2.9]{Petrosyan2012}}]\label{lem:ACF}
    Let $u_\pm$ be a pair of continuous functions such that
    \[
        u_\pm\ge0,\quad\Delta u_\pm\ge0,\quad u_+\cdot u_-=0\quad\text{in }B_1.
    \]
    Then the functional (recall that we write $(x,y) \in \R^d\times\R\cong\R^{d+1}$)
    \[
        r\mapsto \Phi(r)=\Phi(r,u_+,u_-)\coloneqq \frac1{r^4}\int_{B_r}\frac{|\nabla u_+|^2}{|(x,y)|^{d-1}}\,dx\,dy\int_{B_r}\frac{|\nabla u_-|^2}{|(x,y)|^{d-1}}\,dx\,dy
    \]
    is nondecreasing for $0<r<1$.

    Suppose in addition that $\Phi(r_1)=\Phi(r_2)$ for some $0<r_1<r_2<1$. Then the following holds:
    \begin{enumerate}[label=\roman*)]
        \item either $u_+\equiv0$ in $B_{r_2}$ or $u_-\equiv0$ in $B_{r_2}$;
        \item or there exist a unit vector $e \in \R^{d+1}$ and constants $k_\pm>0$ such that
        \[
            u_+(x,y)=k_+((x,y)\cdot e)_+,\quad u_-(x,y)=k_-((x,y)\cdot e)_-\quad\text{in }B_{r_2}.
        \]
    \end{enumerate}
\end{lem}

\begin{proof}[Proof of Proposition~\ref{prop:global solutions}]
    Given a Lipschitz global solution $u$ of \eqref{eq:double top constant coefficients}, we identify it with its even extension with respect to $y$. Note that $u_+\cdot u_-\equiv0$ and, since
    \[
        \Delta u= 2\theta\cH^{d}\mres\{y=0\}\ge0\quad \text{in }\R^{d+1}\setminus\{y=0,u(\cdot,0)\le0\}
    \]
    and $u$ is continuous, we have $\Delta u_+\ge0$. By symmetry we also have $\Delta u_-\ge0$, so the functions $u_\pm$ satisfy the assumptions of Lemma \ref{lem:ACF}.
    Thus, if we define
    \[
        \Phi(r,u)\coloneqq \Phi(r,u_+,u_-)
    \]
    Lemma~\ref{lem:ACF} together with the global Lipschitz regularity of $u$ imply that the following limit exists and is finite:
    \[
        \ell=\lim_{R\to+\infty}\Phi(R,u)\in[0,+\infty)
    \]
    Consider now a blow-down $u_\infty$ of $u$, namely, up to a subsequence,
    \[
        u_R(x,y)\coloneqq \tfrac 1Ru(Rx,Ry)\overset{R\to+\infty}{\longrightarrow} u_\infty(x,y) \quad\text{weakly in $H^1_\loc(\R^{d+1})$ and locally uniformly.}
    \]
    It is easy to see that $u_\infty$ is a Lipschitz global solution of \eqref{eq:double top constant coefficients}, even in the $y$ variable.
    In addition, since $|(x,y)|^{1-d}$ is locally integrable and the functions $u_R$ are equi-Lipschitz, there is a modulus of continuity $\omega_r$ such that
    \[
        \int_{B_r\cap\{|y|<\eps\}}\frac{|\nabla (u_R)_\pm|^2}{|(x,y)|^{d-1}}\,dx\,dy<\omega_r(\eps).
    \]
    Moreover, since $u_R,u_\infty$ are harmonic outside $\{y=0\}$, for all $r,\eps>0$
    \[
        \int_{B_r\cap\{|y|>\eps\}}\frac{|\nabla (u_R)_\pm|^2}{|(x,y)|^{d-1}}\,dx\,dy\to \int_{B_r\cap\{|y|>\eps\}}\frac{|\nabla (u_\infty)_\pm|^2}{|(x,y)|^{d-1}}\,dx\,dy.
    \]
    Combining these two bounds, we  obtain
    \[
        \Phi(r,u_R)\to\Phi(r,u_\infty)\quad\forall\, r>0,\quad \text{as }R\to+\infty.
    \]
    Since $\Phi(Rr,u)=\Phi(r,u_R)$, letting $R\to+\infty$ yields
    \[
        \Phi(r,u_\infty)=\ell\quad\forall\, r>0.
    \]
    Thus, by the second part of Lemma~\ref{lem:ACF}, for any $r>0$:
    \begin{enumerate}[label=\alph*)]
        \item either $u_\infty$ has a sign in $B_{r}$,
        \item or there exist a unit vector $e$ and constants $k_\pm>0$ such that
        \begin{equation}\label{eq:proof acf}
            (u_\infty)_\pm(x,y)=k_\pm((x,y)\cdot e)_\pm\quad\text{in }B_{r}.
        \end{equation}
    \end{enumerate}
    However, case b) is incompatible with \eqref{eq:double top constant coefficients}.
    Indeed, since $u$ is even in the $y$ variable, also the set $\{u>0\}$ must be even with respect to the reflection across the plane $\{y=0\}$.
    However, since the set $\{u>0\}=\{(x,y)\cdot e>0\}$ is a hyperplane, this implies $e\cdot \mathbf e_{d+1}=0$, where $\mathbf e_{d+1}$ is the $(d+1)$-th vector of the canonical Euclidean base of $\R^{d+1}$.
    In particular, $\partial_y u\equiv0$. Since $\partial_y u(\cdot,0)=\theta \sign(u)$ when $u(\cdot,0)\neq0$, the only possibility is that $u(\cdot,0)\equiv0$, in contrast with \eqref{eq:proof acf}.\\
    Since case a) holds, this implies in particular that $\ell=0$.
    Thanks to the monotonicity of $\Phi(r,u)$, this gives
    \[
        \Phi(r,u)\equiv0\quad\forall\, r>0,
    \]
    and therefore (using again
    the second part of Lemma~\ref{lem:ACF}) $u$ has a sign. Assume without loss of generality that $u\ge0$ and consider the function
    \[
        v(x,y) = u-\theta|y|.
    \]
    It follows from \eqref{eq:double top constant coefficients} and the fact that $u\ge0$ that $v$ is a global Lipschitz solution of the thin obstacle problem in $\R^{d+1}$ with $v(0)=0$. Thus, we conclude by applying Lemma~\ref{lem:global lipschitz Signorini}.
\end{proof}

\subsection{Proof of Phase Separation}
We prove Theorem~\ref{teo:phase separation}.
We borrow the following barrier argument from~\cite{Caffarelli2020}.

\begin{lem}\label{lem:barrier lemma}
There exist dimensional constants $\delta,\eta>0$ such that the following holds:

Let $u$ solve~\eqref{eq:double top} in $B_1'\times(0,1)$ and assume that:
\begin{enumerate}[label=\roman*)]
    \item $\|f'\|_\infty<\delta$;
    \item $\varphi(0)=|\nabla\varphi(0)|=0$ and $|D^2\varphi|<\delta$;
    \item $u\le\delta$ in $B_1'$ and $|\partial_y u|\leq \theta+\delta y$ in $B_1'\times(0,1)$;
    \item there is $\hat z\in B_{1/100}'$ satisfying $u(\hat z,0)>\varphi(\hat z)$.
\end{enumerate}
Then
\[
    \partial_yu\ge \tfrac12\theta\quad\text{in}\quad\left\{y=\eta,|x|\le 1/2\right\}.
\]
\end{lem}
\begin{proof}
    Up to considering $u(x,y)-\varphi(x)$ and relabeling $\delta$, we can assume that $u$ satisfies
    \[\begin{cases}
        \Delta u=g(x)   &\text{in }B_1'\times(0,1),\\
        |\partial_yu|\le\theta  &\text{in }\{y=0\},\\
        \partial_yu=\sign(u)f(|u|)  &\text{in }\{y=0,u\neq0\}
    \end{cases}\]
    where 
    \[
        |g|\le\delta,\qquad u\le \delta\quad\text{in }B_1',\qquad|\partial_yu|\le\theta+\delta y\quad\text{in }B_1'\times(0,1).
    \]
    Given $M\ge1$ we define
    \[
        V(x,y)\coloneqq u(x,y)-M\delta|x-\hat z|^2 + (dM+1/2)\delta y^2-(\theta-\delta)y,
    \]
    and for $0<\eta\le1/100$ we set $\Gamma\coloneqq B_1'\times[0,\eta]$. Since
    \[
        \Delta V = \Delta u+\delta\ge0\quad\text{in }\Gamma,
    \]
    by the maximum principle there is $(\bar x,\bar y)\in\partial\Gamma$ such that
    \begin{equation}\label{eq:proof-barrier-phaseSeparation-maxPrinc}
        V(\bar x,\bar y) = \max_{\partial \Gamma} V = \max_\Gamma V \ge V(\hat z,0)= u(\hat z,0)>0.
    \end{equation}
    Provided $M$ is large enough and $\delta,\eta$ are small enough, we show that 
    \begin{equation}\label{eq:proof barrier maximizer}
        \bar y=\eta\quad\text{and}\quad |\bar x|\le1/2.
    \end{equation}
    We argue by contradiction, splitting the proof in two cases.

    \noindent $\bullet$  \textit{$\bar y=0$.} In this case, by \eqref{eq:proof-barrier-phaseSeparation-maxPrinc}, we have
    \[
        0<V(\bar x,0)=u(\bar x,0)-M\delta|\bar x-\hat z|^2\le u(\bar x,0),
    \]
    which implies
    \[
        \partial_yu(\bar x,0)=f(u(\bar x,0))\ge \theta -\|f'\|_\infty\sup_{B_1'} u\ge\theta -\delta^2.
    \]
    On the other hand,
    \[
        0\ge\partial_y V(\bar x,0) = \partial_y u(\bar x,0) -\theta + \delta,
    \]
    so we reach a contradiction for $\delta$ small enough.

    \noindent $\bullet$  \textit{$|\bar x|\geq 1/2$.}
    Since
    \[
        0<V(\bar x,\bar y) = u(\bar x,\bar y)-M\delta|\bar x-\hat z|^2+(dM+1/2)\delta\bar y^2 - (\theta -\delta)\bar y,
    \]
    recalling that $M\ge1, |\hat z|\le 1/100$, and $\bar y\le\eta$, we have
    \begin{multline*}
        u(\bar x,\bar y)> M\delta|\bar x-\hat z|^2 - (dM+1/2)\delta\bar y^2+(\theta-\delta)\bar y\\
        > (\tfrac M9-(dM+1/2)\eta^2-\eta)\delta +\theta\bar y
        \ge \tfrac M{10}\delta+ \theta\bar y
    \end{multline*}
    provided $\eta$ is small enough.
    On the other hand, since $|\partial_yu|\le \theta+\delta y$ and $u\le \delta$ in $\{y=0\}\cap B_1$, we get
    \[
        u(\bar x,\bar y)\le u(\bar x,0)+\theta\bar y + \tfrac\delta2\bar y^2\le (1+\eta^2/2)\delta+ \theta\bar y,
    \]
    which is impossible for $M$ large enough.
    Hence, we proved that \eqref{eq:proof barrier maximizer} holds.

  Now, we deduce from~\eqref{eq:proof-barrier-phaseSeparation-maxPrinc} and \eqref{eq:proof barrier maximizer} that
    \[
        0\le\partial_y V(\bar x,\eta) = \partial_y u(\bar x,\eta) +(2dM+1)\delta\eta - (\theta - \delta),
    \]
    therefore
    \begin{equation}\label{eq:barrierLemma-phaseSep-touchingPoint}
        \theta-\partial_y u(\bar x,\eta)\le C\delta.
    \end{equation}
    By assumption the function $\theta+2\delta\eta-\partial_yu$ is harmonic and positive in $B_1'\times(0,2\eta)$. Thus, by Harnack inequality, there exists $C>0$ such that, setting $\Gamma'=B_{1/2}'\times[\eta/2,3\eta/2]$, we have
    \[
    \sup_{\Gamma'} \big(\theta+2\delta\eta-\partial_yu\big) \le C\inf_{\Gamma'} \big(\theta+2\delta\eta-\partial_yu\big).
    \]
    Since $(\bar x,\eta)\in\Gamma'$ by \eqref{eq:proof barrier maximizer}, using~\eqref{eq:barrierLemma-phaseSep-touchingPoint} we have
    \[
    \sup_{\Gamma'} \big(\theta+2\delta\eta-\partial_yu\big)\le C\delta,
    \]
    therefore
    \[
        \partial_yu\ge \theta -C\delta\ge\tfrac12\theta\quad\text{ in }\big\{y=\eta,|x|\le 1/2\big\}
    \]
    provided $\delta$ is small enough, as we wanted.
\end{proof}

\begin{proof}[Proof of Theorem~\ref{teo:phase separation}]
    We argue by contradiction. Assume there are solutions $u_k$ with $u_k(0)=\varphi_k(0)$ and satisfying
    \[
        \|u_k\|_{L^2(B_1^+)}+\|\varphi_k\|_{C^{1,1}}+\|f_k'\|_\infty\le C,
    \]
    but such that there are $\rho_k\to0$ and $x_k,y_k\in B_{\rho_k}'$ with $u_k(x_k,0)>\varphi_k(x_k)$ and $u_k(y_k,0)<\varphi_k(y_k)$.
    Up to considering $u_k-\varphi_k(0)-x\cdot\nabla\varphi_k(0)$ we can assume $\varphi_k(0)=|\nabla\varphi_k(0)|=0$.

    Set $u_{k,r}(x,y)\coloneqq \frac1ru_k(rx,ry)$ for $r\in(0,1)$. We show that
    \begin{equation}\label{eq:proof phase separation}
        |u_{k,r}(\cdot,0)|\le \omega(r)\quad\text{in }B_1'
    \end{equation}
    for some modulus of continuity $\omega$ independent of $k$. Indeed, by Lemma~\ref{lem:lip double top}  there exists $C>0$ independent from $k,r$ such that $\|\nabla u_{k,r}\|_{L^\infty(B_{1/2r}^+)}\le C$.
    It follows that, up to a subsequence, if $r_k\to0$ then
    \[
        u_{k,r_k}\overset{k\to+\infty}{\longrightarrow} u_{\infty}\quad\text{locally uniformly in }\R^d\times[0,+\infty).
    \]
    As the functions $u_{k,r_k}$ solve \eqref{eq:double top} with $f_{k,r_k}=f_k(r_k\cdot)\to \theta$ and obstacle $\varphi_{k,r_k} = \frac 1{r_k}\varphi_k(r_k\cdot)\to 0$, we deduce that $u_{\infty}$ is a Lipschitz global solution of \eqref{eq:double top constant coefficients}.
    By Proposition~\ref{prop:global solutions} we deduce $u_{\infty}(\cdot,0)\equiv0$, and \eqref{eq:proof phase separation} follows arguing by contradiction.
    
    Thus, choosing first $r$ small enough and then $k$ sufficiently large the functions $\pm u_{k,r}$ satisfy the assumptions of Lemma~\ref{lem:barrier lemma}. This yields
    \[
        \pm\partial_yu_k\ge \tfrac\theta2\quad\text{in }\{y=r\eta,|x|\le r/2\}
    \]
    for $k$ large enough, a contradiction.
\end{proof}

\section{Generic Regularity for a Two-phase Thin Obstacle}\label{sec:double top}
Let $\theta>0$ be a fixed parameter. We will consider differentiable obstacles $\varphi$ on $[-1,1]$ satisfying the analog of~\eqref{eq:derivativeBound-discrete}, namely
\begin{equation}\label{eq:derivativeBound-discrete2}
    |\varphi'(x)|\le C\big(|\log (x+1)|+|\log(1-x)|\big).
\end{equation}
We want to show a generic regularity result for~\eqref{eq:ele_discrete2}.
To this aim, we assume $K=[-1,1]$ and we consider solutions in $B_2$ of
\begin{equation}\label{eq:double top 2d}\begin{cases}
    -\Delta u=0 &\text{in }B_2\setminus \{x_2=0\},\\
    \partial_2u=\theta  &\text{in }\{x_2=0,|x_1|>1\},\\
    |\partial_2u|\le\theta    &\text{in }\{x_2=0,|x_1|\le 1\},\\
    \partial_2 u = \sign(u-\varphi)\theta   &\text{in }\{u\neq\varphi,|x_1|\le1,x_2=0\},\\
    u(x_1,-x_2)=u(x_1,x_2)  &\text{for }(x_1,x_2)\in\R^2,
\end{cases}\end{equation}
where $u$ is even in $x_2$ and the value $\partial_2u$ at $\{x_2=0\}$ is intended as the limit from the right, namely $$\partial_2u(x_1,0)=\lim_{t\to0^+}\frac{u(x_1,t)-u(x_1,0)}{t}.$$
We remark that $u$ solves~\eqref{eq:ele_discrete} in $B_2$ with $K=[-1,1]$ if and only if $\frac2\pi u+\theta |x_2|$ solves~\eqref{eq:double top 2d}.
Moreover, if $u$ solves~\eqref{eq:double top 2d} and $u\ge\varphi$ then $u-\theta |x_2|$ solves \eqref{eq:top}.
Finally, \eqref{eq:double top 2d} corresponds to \eqref{eq:double top} with $f\equiv\theta$ and $d=1$.

In this section we exploit the results from Appendices~\ref{sect:TOP} and \ref{sec:phase separation} to show optimal regularity and a generic regularity result for \eqref{eq:double top 2d}.

\subsection{Optimal regularity} 
Observe that the form of the obstacle~\eqref{eq:derivativeBound-discrete2} (namely, its low regularity at $\pm1$) allows for non-Lipschitz solutions at $(\pm1,0)$ (see the Remark below). Nevertheless, it follows from the results of Appendix~\ref{sec:phase separation} that solutions are Lipschitz in the interior of $K$.
\begin{oss}
    Consider $u=-\frac{\theta}{2\pi}\log|\cdot|*\chi_{(-1,1)}$. Then $u$ is a non-Lipschitz solution of~\eqref{eq:double top 2d} with $\varphi = u$ in $[-1,1]$. Also, $\varphi$ satisfies \eqref{eq:derivativeBound-discrete2} in $(-1,1)$. Note that Theorem~\ref{teo:generic_doubleTOP} implies that this behavior is not generic.
\end{oss}
In view of the previous remark, to prove optimal regularity we localize the problem considering solutions in $B_1$ of
\begin{equation}\label{eq:doubletop-localised}\begin{cases}
    -\Delta u=0 &\text{in }B_1\setminus\{x_2=0\},\\
    |\partial_2u|\le\theta    &\text{in }\{x_2=0\},\\
    \partial_2 u = \sign(u-\varphi)\theta   &\text{in }\{x_2=0,u(\cdot,0)\neq\varphi\}.\\
\end{cases}\end{equation}
As this corresponds to \eqref{eq:double top} with $f\equiv\theta$ and $d=1$, the results from Appendix~\ref{sec:phase separation} apply.
\begin{cor}\label{cor:optimal-regularity-doubleTop}
    Let $u$ be a solution of~\eqref{eq:doubletop-localised} with $\varphi\in C^{1,1}((-1,1))$. Then $u(\cdot,0)\in C^{1,1/2}_\loc((-1,1))$. More precisely, there exists $C>0$, depending only on $\theta,\|u\|_{L^2(B_1)}$, and $\|\varphi\|_{C^{1,1}}$, such that $\|u\|_{C^{1,1/2}(B_{1/2}^\pm)}\leq C$.
\end{cor}
\begin{proof}
    As $u-\theta|x_2|$ solves the thin obstacle problem \eqref{eq:top} in $B_r(x)$ when $u(\cdot,0)\ge\varphi$ in $B_r'(x)$, the result follows from Lemma~\ref{lem:regularity-top} and Theorem~\ref{teo:phase separation}.
\end{proof}

\subsection{Generic regularity} We consider a continuous family of solutions $\{u_t\}_{t\in[-1,1]}$ of \eqref{eq:double top 2d} in $B_2$, strictly increasing in the following sense:
\begin{equation}\label{eq:strictMonotonicity_HP_doubleTOP}
    \begin{cases}u_{t+\eps}\ge u_t      &\text{in }\partial B_2,\\
    u_{t+\eps}\ge u_t + \eta\eps        &\text{in }\partial B_2\cap\{|x_2|>1\},
\end{cases}\end{equation}
for some $\eta>0$ and every $-1\le t<t+\eps\le 1$.

We define regular solutions for \eqref{eq:double top 2d} in analogy with Definition~\ref{defn:regular-solution-top}. We remark that, due to the low regularity at $\pm1$ of the obstacle $\varphi$, we also include contact points in the singularities of $\varphi$.
\begin{defn}\label{defn:regular solution double top}
Given $x_0\in\{u(\cdot,0)=\varphi\}$ we say that $x_0\in\Sing^+(u)$ (respectively $x_0\in\Sing^-(u)$) if  $u(\cdot,0)\ge\varphi$ (resp. $u(\cdot,0)\le\varphi$) in a neighbourhood of $x_0$ and it is a singular point (in the sense of~\eqref{eq:sing-points}) for the function $u-\theta|x_2|$ (resp. $-u-\theta|x_2|$).\\
We say that $x_0$ is singular and we write $x_0\in\Sing(u)$ if:\\
- either $x_0\in \Sing^+(u)\cup\Sing^-(u)$;\\
- or if $x_0\in\{\pm1\}$ and $u(x_0,0)=\varphi(x_0)$.\\
We say that $u$ solving \eqref{eq:double top 2d} is \emph{regular} if there are no singular points.
\end{defn}

Note that
\begin{equation}\label{eq:observation singular points}
    \partial_2u(x_0,0)=\theta>0\quad\forall\, x_0\in\Sing^+(u)
\end{equation}
and similarly for $x_0\in\Sing^-(u)$.

\begin{teo}\label{teo:generic_doubleTOP}
    Assume $\varphi\in C^{2,\alpha}$ and let $u_t$ be a continuous family of solutions of~\eqref{eq:double top 2d} satisfying~\eqref{eq:strictMonotonicity_HP_doubleTOP}.
    Then, for a.e. $t\in[-1,1]$, $u_t$ is regular in the sense of Definition~\ref{defn:regular solution double top}.
\end{teo}
We will make use of the following fact.
\begin{lem}[Convergence of singular points]\label{lem:stability singular doubleTOP}
    Let $u_k,u$ solve~\eqref{eq:doubletop-localised} in $B_1$ with obstacles $\varphi_k,\varphi\in C^{1,1}$.
    Assume that:
    \begin{itemize}
        \item[\textit{i)}] $u_k\to u, \varphi_k\to \varphi$ uniformly, and $\sup_{k}\|\varphi_k\|_{C^{1,1}((-1,1))}<+\infty$;
        \item[\textit{ii)}] there are points $x_k\in\Sing(u_k)$ with $x_k\to x_\infty\in B_{1/2}$. 
    \end{itemize}
        Then $x_\infty\in\Sing(u)$.
\end{lem}
\begin{proof}
    We can assume without loss of generality that $x_k\in\Sing^+(u_k)$.
    It follows from Corollary~\ref{cor:optimal-regularity-doubleTop} and \eqref{eq:observation singular points} that, for $k$ large enough, the functions $u_k-\theta |x_2|$ and $u-\theta|x_2|$ solve~\eqref{eq:top} in $B_{\bar r}(x_\infty)$ for some $\bar r$ independent of $k$. We conclude by Lemma~\ref{lem:stability singular top}
\end{proof}

\begin{proof}[Proof of Theorem~\ref{teo:generic_doubleTOP}]
    We first show that for each point $a\in\{\pm1\}$ there is at most one value of $t$ such that $u_t(a,0)=\varphi(a)$.
    We argue by contradiction, assuming that there are two values $t>t'$ such that $u_t(a,0)=u_{t'}(a,0)=\varphi(a)$.
    We can assume $a=0$ and $K=[-2,0]$. Since $u_t\ge u_{t'}$ and both are harmonic in $\R^2\setminus K$, by the strong maximum principle the function $w\coloneqq u_t-u_{t'}$ solves
    \[\begin{cases}
        \Delta w = 0    &\text{in }B_\delta\setminus\{x_2=0,x_1\le0\},\\
        w\ge0           &\text{in }B_\delta,\\
        w > 0           &\text{in }B_\delta\setminus\{x_2=0,x_1\le0\}\\
    \end{cases}\]
    and $w(0,0)=0$. It follows that there is $\eps>0$ sufficiently small such that $w\ge \eps U_0$ in $B_{\delta/4}$, where $U_0(\rho,\theta)=\rho^{1/2}\cos\theta/2$.
    On the other hand, Lemma~\ref{lem:normal Lip} gives $|\partial_2u_t|+|\partial_2u_{t'}|\le C$ close to $0$, so $w(0,x_2)\le C x_2$ for $x_2$ small enough.
    This implies $|x_2|^{1/2} \sim U_0(0,x_2)\le Cx_2$ for $x_2$ small enough, which is a contradiction.

    Thus, it is sufficient to show the result assuming
    \[
        \cup_t\{u_t(\cdot,0)=\varphi\}\ssubset(-1,1)\quad\text{and}\quad \Sing(u_t)=\Sing^+(u_t).
    \]
    Let $t_0\in \{t : \Sing^+(u_t)\neq\emptyset\}$ be a point with density 1. Given $\bar r>0$, by Lemma~\ref{lem:stability singular doubleTOP} there is $\delta=\delta(\bar r)>0$ such that if $|s-t_0|<\delta$ then $\Sing^+(u_s)$ is contained in a $\bar r/2$-neighbourhood of $\Sing^+(u_{t_0})$.
    As $\Sing^+(u_{t_0})$ is compact and $t_0$ has density 1, we can find $x_0\in\Sing^+(u_{t_0})$ such that
    \[
        \cH^1\big(\{s\in(t_0-\delta,t_0+\delta): \Sing^+(u_s)\cap(x_0-\bar r/2,x_0+\bar r/2)\neq\emptyset\}\big)>0.
    \]
    By Corollary~\ref{cor:optimal-regularity-doubleTop} the family $(u_t)$ is continuous in the $C^1(B_{1-\sigma}^+)$ topology, for any $\sigma>0$. As $\partial_2u_{t_0}(x_0,0) = \theta>0$ by \eqref{eq:observation singular points}, up to making $\bar r,\delta$ smaller we have $\partial_yu_s(x,0)\ge0$ for all $x\in(x_0-\bar r,x_0+\bar r)$ and all $s\in(t_0-\delta,t_0+\delta)$. Thanks to \eqref{eq:double top 2d} this implies $u_s(\cdot,0)\ge\varphi$ in $(x_0-\bar r,x_0+\bar r)$ for $s\in(t_0-\delta,t_0+\delta)$.
    We deduce that the functions $v_s\coloneqq u_s-\theta |x_2|$,  satisfy the assumptions of Theorem~\ref{teo:genericTOP} in $B_{\bar r}(x_0)$, reaching a contradiction.
\end{proof}

\subsection{Applications to Section \ref{sec:min obst}}
Following Appendix~\ref{sect:TOP}, we note explicitly the following applications of the results in this section in the context of Theorem~\ref{teo:generic_discreteCase2}. We state these results with the potentials extended in the whole $\R$, setting them equal to $+\infty$ outside of $K$.

\begin{lem}\label{lem:equivalence regular definitions discrete}
    Suppose that a positive measure $\mu$ is such that
    \[
        u_\mu(x) \coloneqq -\log|\cdot|*\mu(x) + C,\quad x\in\R^2
    \]
    solves \eqref{eq:ele_discrete} for some constant $C\in\R$, and some potential $V\colon \R\to\R\cup\{+\infty\}$ satisfying \eqref{eq:derivativeBound-discrete}. Then the couple $(u_\mu,\mu)$ satisfies the properties of Definition \ref{defn:regular-potential-discrete} if and only if $\tfrac2\pi u_\mu+\theta|x_2|$ is a regular solution in the sense of Definition~\ref{defn:regular solution double top}.
\end{lem}

\begin{lem}\label{lem:stability singular doubleTOP applied}
    Suppose that $u_j,u$ are of the form
    \[
        u_j = -\log|\cdot|*\mu_j + C_j,\quad u = -\log|\cdot|*\mu + C
    \]
    for some nonnegative measures $\mu_j,\mu$ and constants $C_j,C$. Assume in addition that they solve \eqref{eq:ele_discrete} with potentials $V_j,V\colon\R\to\R\cup\{+\infty\}$ and, setting $K_j = \{V_j<+\infty\}, K = \{V<+\infty\}$:
    \begin{itemize}
        \item[\textit{i)}] $u_j\to u$ locally uniformly, the graphs of $V_j$ over $K_j$ converge to the graph of $V$ over $K$ in the hausdorff distance, and $\sup_{j}\|V_j\|_{C^{1,1}_\loc(\overset{\circ}{K})}<+\infty$;
        \item[\textit{ii)}] $(u,\mu)$ satisfies the properties of Definition~\ref{defn:regular-potential-discrete}.
    \end{itemize}
    Then $(u_j,\mu_j)$ satisfies Definition~\ref{defn:regular-potential-discrete} for $j$ large enough.
\end{lem}

\begin{cor}\label{cor:generic doubleTOP}
    Let $V\in C_{\rm loc}^{2,\alpha}\left(\bigcup_h (a_h,b_h)\right)$ satisfy~\eqref{eq:derivativeBound-discrete}, and let $(u_t)_{t>0}$ be a continuous family of functions of the form
    \[
        u_t = -\log|\cdot|*\mu_t + C_t
    \]
    for some nonnegative measures $\mu_t$ and constants $C_t$,
    solving \eqref{eq:ele_discrete} and strictly monotone in the sense of Proposition~\ref{prop:monotonicity_discrete}. Then for a.e. $t>0$ the couple $(u_t,\mu_t)$ satisfies the properties of Definition~\ref{defn:regular-potential-discrete}.
\end{cor}

\section{Riesz potentials}\label{app:riesz case}

Given an integer $d\ge1$ and a potential $V\colon\R^d\to[-\infty,+\infty)$, our strategy also applies to minimizers of the energies
\[\cE(\mu)=\iint_{\R^d\times\R^d}\mathfrak g(x-y)\,d\mu(x)\,d\mu(y) + 2\int_{\R^d}V(x)\,d\mu(x),\]
where $\mathfrak g$ is as in \eqref{eq:g}.
Indeed, existence and uniqueness of minimizing measures hold under mild assumptions on $V$, see~\cite[Chapter 2]{serfaty2024lectures}.
Also, similarly to Section~\ref{sect:min to thin} (cp. \cite[Proposition 2.15 and Section 2.4]{serfaty2024lectures}), given $\mu_V$ the minimizing measure we can associate the function
\[
    h^{\mu_V}(x)\coloneqq \mathfrak g*\mu_V(x),\qquad x\in\R^d.
\]
Then, there is a constant $c_V\in\R$ such that the function $u^{\mu_V}\coloneqq h^{\mu_V}-c_V$ solves
\begin{equation}\label{eq:fractional_obstacle}\begin{cases}
    u^{\mu_V}\ge -V &\text{in }\R^d,\\
    u^{\mu_V}=-V    &\text{in }\supp(\mu_V),\\
    (-\Delta)^{\frac{d-\sigma}2}u^{\mu_V}= c_{d,\sigma}\mu_V   &\text{in }\R^d,
\end{cases}\end{equation}
for some constant $c_{d,\sigma}$. Furthermore, the converse is true: a probability measure $\mu$ is minimizing if there is a solution $u^\mu$ of~\eqref{eq:fractional_obstacle} with $(-\Delta)^{\frac{d-\sigma}2}u^{\mu}= c_{d,\sigma}\mu$.

Also in this context, if $u^{\mu_V}$ is a regular solution of~\eqref{eq:fractional_obstacle} (in a sense analogous to Definition~\ref{defn:regular-solution-top}) then $\mu_V$ satisfies condition \textbf{($\mathbf{A_\sigma}$)}. Thus, thanks to~\cite[Theorem 1.1]{CarducciColombo} (see also~\cite{FernandezRealRosOton,FERNANDEZREAL2023109323}), the same strategy adopted to prove Theorem~\ref{teo:generic_continuousCase} yields the following:
\begin{teo}\label{teo:generic_riesz}
    Given an integer $1\le d\le3$, a real number $\sigma \in (d-2,d)$, a potential $V\in C^{4,\alpha}_\loc(\R^d)$ for some $\alpha\in(a^-,1]$ where $a=1+\sigma-d$, and $\gamma\in\R$, consider the family $V_{\gamma,s}:=s^{\gamma \sigma-1}V(s^\gamma\,\cdot)$, $s>0$, and let $\mu_{V_{\gamma,s}}$ denote the associated minimizing probabilities.
    Then ${V_{\gamma,s}}$ is regular for a.e.  $s>0$.
\end{teo}

The assumption $d\le3$ comes from the same restriction in \cite[Theorem 1.1]{CarducciColombo}.

\begin{oss}
Thanks to~\cite{jhaveri2017higher,Abatangelo2020}, higher regularity of the potential $V$ yields higher regularity on the density of $\mu_{V_{\gamma,s}}$ and the geometry of the support. More precisely, if $V$ is a regular potential of class $C^{k+{\scriptstyle\frac{d-\sigma}{2}}+\alpha}(\R^d)$ for $k>2$ and $\alpha\in(0,1)$ with $\alpha\pm {\scriptstyle\frac{d-\sigma}{2}}\not\in\N$,\footnote{As in the introduction, we denote
 $$C_{\rm loc}^{k+\frac{d-\sigma}{2}+\beta}(\R^d)=
 \left\{
 \begin{array}{ll}
 C_{\rm loc}^{k,\beta+\frac{d-\sigma}{2}}(\R^d)&\text{if }\beta \leq 1-\frac{d-\sigma}{2},\\
 C_{\rm loc}^{k+1,\beta-\frac{d-\sigma}{2}}(\R^d)&\text{if }\beta > 1-\frac{d-\sigma}{2}.
 \end{array}
 \right.
 $$}
 then
$\mu_{V}$ is supported over finitely many disjoint compact sets 
$\{K_j\}_{1\leq j \leq M}\subset \R^d$, with $\partial K_j$ a $(d-1)$-dimensional manifold of class $C^{k,\alpha}$. Also, 
the function $Q_V$ in \eqref{eq:form-regular-potential sigma} is of class $C^{k-1,\alpha}(K_j)$ in a neighborhood of $\partial K_j$.
\end{oss}

\bibliographystyle{aomalpha}
\bibliography{Bibl}
\end{document}